%% file: doublemaps.tex
\documentclass[a4paper]{amsart}
\usepackage{amssymb,amscd}
\usepackage{amsmath,amsfonts,latexsym}

\usepackage{pstricks,epsfig}
\usepackage[english]{babel}
\usepackage[totalwidth=16cm,totalheight=22cm]{geometry}
\usepackage{graphicx}
\usepackage{psfrag}
\usepackage{tikz}
\usetikzlibrary{matrix,arrows}
\usepackage{multicol}
\input{macros.tex}

\begin{document}

\title[Symplectic Wick rotations between moduli spaces of 3-manifolds]{Symplectic Wick rotations between moduli spaces of 3-manifolds}
\author{Carlos Scarinci}
\author{Jean-Marc Schlenker}
\thanks{J.-M. S. acknowledges support from U.S. National Science Foundation grants DMS 1107452, 1107263, 1107367 ``RNMS: GEometric structures And Representation varieties'' (the GEAR Network).}
\address{CS: University of Nottingham, School of Mathematical Sciences, University Park, NG7 2RD, Nottingham, UK}
\email{carlos.scarinci@nottingham.ac.uk}
\address{JMS: University of Luxembourg, Campus Kirchberg,
Mathematics Research Unit, BLG 6, rue Richard Coudenhove-Kalergi,
L-1359 Luxembourg, Luxembourg}
\email{jean-marc.schlenker@uni.lu}

\date{\today (v1)}

\begin{abstract}
Given a closed hyperbolic surface $S$, let $\cQF$ denote the space of quasifuchsian hyperbolic
metrics on $S\times\R$ and $\cGH_{-1}$ the space of maximal globally hyperbolic anti-de Sitter
metrics on $S\times\R$. We describe natural maps between (parts of) $\cQF$ and $\cGH_{-1}$,
called ``Wick rotations'', defined in terms of special surfaces (e.g. minimal/maximal surfaces,
CMC surfaces, pleated surfaces) and prove that these maps are at least $C^1$ smooth and symplectic
with respect to the canonical symplectic structures on both $\cQF$ and $\cGH_{-1}$. Similar results
involving the spaces of globally hyperbolic de Sitter and Minkowski metrics are also described.

These 3-dimensional results are shown to be equivalent to purely 2-dimensional ones. Namely, consider
the double harmonic map $\cH:T^*\cT\to\cTT$, sending a conformal structure $c$ and a holomorphic quadratic
differential $q$ on $S$ to the pair of hyperbolic metrics $(m_L,m_R)$ such that the harmonic maps isotopic
to the identity from $(S,c)$ to $(S,m_L)$ and to $(S,m_R)$ have, respectively, Hopf differentials equal
to $i q$ and $-i q$, and the double earthquake map $\cE:\cT\times\cML\to\cTT$, sending a hyperbolic metric
$m$ and a measured lamination $l$ on $S$ to the pair $(E_L(m,l), E_R(m,l))$, where $E_L$ and
$E_R$ denote the left and right earthquakes. We describe how such 2-dimensional double maps
are related to 3-dimensional Wick rotations and prove that they are also $C^1$ smooth and symplectic.
\end{abstract}

\maketitle

\tableofcontents

\input{doublemaps1} % intro and results

\input{doublemaps2} % background

\input{doublemaps3} % wick rotations and double maps

\input{doublemaps4} % regularity of the earthquake map

\input{doublemaps5} % the double earthquake map is symplectic

\input{doublemaps6} % Constant mean curvature surfaces

\input{doublemaps7} % Flat and de Sitter spacetimes

\bibliographystyle{amsplain}
\bibliography{biblio}

\end{document}

%% file: macros.tex
\newtheorem{cor}{Corollary}[section]
\newtheorem{theorem}[cor]{Theorem}
\newtheorem{prop}[cor]{Proposition}
\newtheorem{lemma}[cor]{Lemma}

\theoremstyle{definition}
\newtheorem{defi}[cor]{Definition}
\theoremstyle{remark}
\newtheorem{remark}[cor]{Remark}

\newcommand{\cAF}{{\mathcal {AF}}}

\newcommand{\cC}{{\mathcal C}}
\newcommand{\cD}{{\mathcal D}}
\newcommand{\cE}{{\mathcal E}}

\newcommand{\cH}{{\mathcal H}}
\newcommand{\cL}{{\mathcal L}}
\newcommand{\cM}{{\mathcal M}}

\newcommand{\cS}{{\mathcal S}}
\newcommand{\cT}{{\mathcal T}}
\newcommand{\cTb}{{\overline{\mathcal T}}}
\newcommand{\cQ}{{\mathcal Q}}
\newcommand{\cW}{{\mathcal W}}
\newcommand{\cX}{{\mathcal X}}

\newcommand{\cCP}{{\mathcal C\mathcal P}}
\newcommand{\cML}{{\mathcal M\mathcal L}}
\newcommand{\cTML}{{\mathcal T\times \mathcal M\mathcal L}}
\newcommand{\cTT}{{\mathcal T\times \overline{\mathcal T}}}

\newcommand{\cGH}{{\mathcal G\mathcal H}}
\newcommand{\cQF}{{\mathcal Q\mathcal F}}

\newcommand{\C}{{\mathbb C}}

\newcommand{\R}{{\mathbb R}}

\newcommand{\Et}{\tilde{E}}

\newcommand{\psl}{\mathfrak{sl}}

\newcommand{\PSL}{\mathrm{PSL}}

\newcommand{\maxi}{\mathrm{max}}

\newcommand{\cotan}{\mbox{cotan}}

\newcommand{\tr}{\mbox{\rm tr}}

\newcommand{\isom}{\mbox{isom}}
\newcommand{\hopf}{\mbox{Hopf}}

\newcommand{\cmc}{\mathrm{CMC}}

\newcommand{\II}{I\hspace{-0.1cm}I}
\newcommand{\III}{I\hspace{-0.1cm}I\hspace{-0.1cm}I}

\newcommand{\omom}{\omega_{WP}\oplus \overline{\omega_{WP}}}

\newcommand{\omegawpb}{\overline{\omega_{WP}}}

\newcommand{\Rep}{\mathrm{Rep}}

\newcommand{\bbS}{\mathbb{S}}
\newcommand{\bbH}{\mathbb{H}}

\newcommand{\id}{\mathrm{id}}
\newcommand{\Ad}{\mathrm{Ad}}

\newcommand{\re}{\mathrm{Re}}
\newcommand{\cG}{\mathcal{G}}
\newcommand{\cHE}{\mathcal{HE}}
\newcommand{\Riem}{\mathrm{Riem}}
\newcommand{\Hom}{\mathrm{Hom}}
\newcommand{\Ric}{\mathrm{Ric}}

%% file: doublemaps1.tex
\section{Introduction and results}
\label{sc:intro}

\subsection*{Notations}

In this paper we consider a closed, oriented surface $S$ of genus $g\geq 2$
% with fundamental group $\Gamma$,
and the 3-dimensional manifold $M=S\times \R$. 
The boundary of $M$ is the disjoint union of two surfaces homeomorphic to $S$,
which we denote by $\partial_+M$ and $\partial_-M$.

We denote by $\cT$ the 
Teichm\"uller space of $S$, which is considered either as the space of conformal
structures, the space of complex structures compatible with the orientation, or the space
of hyperbolic metrics on $S$, all considered up to isotopy, and by $\cTb$ the 
Teichm\"uller space of $S$ with the opposite orientation. Recall that $\cT$
is naturally endowed with a symplectic form $\omega_{WP}$, called the Weil-Petersson 
symplectic form, and $\cTb$ has the corresponding symplectic form $\omegawpb$
(which differs from $\omega_{WP}$ by a sign).

We also denote by $\cML$ the space of measured laminations on 
$S$ (a brief definition is recalled below), by 
$\cL$ the space of (unmeasured) laminations, and by $\cCP$ the
space of complex projective structures on $S$, considered up to isotopy. 
Recall that $\cCP$ is diffeomorphic to a ball of real dimension $12g-12$, that 
is, twice the dimension of $\cT$. We will use the notation $\cQ$ for the
space of holomorphic quadratic differentials on $S$, and, given a complex structure 
$c\in \cT$, by $\cQ_c$ the fiber of $\cQ$ over $S$, that is, the vector space of 
holomorphic quadratic differentials on $(S,c)$.

\subsection{Quasifuchsian hyperbolic manifolds}

The first moduli space we will consider here is the space of quasifuchsian hyperbolic metrics 
on $M=S\times\R$, which can be defined in terms of convex subsets. Given a hyperbolic metric $h$ on $M$, we say that a subset $K\subset M$
is convex if any geodesic segment in $M$ with endpoints in $K$ is contained in $K$. 

\begin{defi}\label{df:quasi}
A complete hyperbolic metric $h$ on $M$ is called quasifuchsian if $(M,g)$ contains
a non-empty compact convex subset. We denote by $\cQF$ the space of quasifuchsian hyperbolic metrics on $M$, considered up to isotopy.
\end{defi}

Note that there are other equivalent definitions
of quasifuchsian 
%manifolds, e.g. as quotients of the $3$-dimensional hyperbolic space
%by a Kleinian group $\Gamma$ with limit set a Jordan curve $\Lambda_\Gamma$.
representations, e.g. as quasiconformal deformations of Fuchsian representations.

Let $\cX(\pi_1S,\PSL(2,\C))$ denote the character variety of representations of the fundamental
group of $S$ in $\PSL(2,\C)$. Given a quasifuchsian metric $h\in\cQF$, we can consider
its holonomy representation $\rho\in\cX(\pi_1S,\PSL(2,\C))$. This gives a natural map 
$hol:\cQF\to \cX(\pi_1S,\PSL(2,\C))$ which is a local diffeomorphism (but is neither surjective nor injective, see e.g. \cite{dumas-survey}).

The character variety $\cX(\pi_1S,\PSL(2,\C))$ is known to be equipped with a complex symplectic
structure $\omega_G$, obtained by taking the cup-product of the cohomology classes
corresponding to infinitesimal deformations of a representation, 
see \cite{goldman-symplectic}. Pulling back $\omega_G$ by the holonomy map $hol$ gives a complex symplectic
structure on $\cQF$, which we also call $\omega_G$. We will denote by 
$\omega_G^i$ the imaginary part of $\omega_G$, which is a (real) symplectic structure
and will play a key role in what follows.

% We will see in Section \ref{sc:background} that this symplectic form on $\cQF$ also appears under different forms.

\subsection{Globally hyperbolic anti-de Sitter manifolds}
\label{ssc:ghads}

The second moduli space of interest in this work is that of globally hyperbolic
maximal anti-de Sitter metrics on $M=S\times\R$.

The $3$-dimensional anti-de Sitter space, denoted here by $AdS^3$, 
can be defined as the quadric
$$ \{ p\in \R^{2,2}~|~\langle p,p\rangle=-1\} $$
with the induced metric from the metric of signature $(2,2)$ on $\R^4$. 

\begin{defi} \label{df:ads}
A globally hyperbolic maximal (GHM) anti-de Sitter (AdS) metric on $M$ is a 
Lorentzian metric $g$ on $M$ such that
\begin{itemize}
\item  $M$ is locally modeled on $AdS^3$
\item it contains a Cauchy surface homeomorphic to $S$
\item it is maximal under these conditions.
\end{itemize}
We call $\cGH_{-1}$ the space of GHM AdS metrics on $M$, considered up to isotopy.
\end{defi}

We say that a surface $\Sigma\subset M$ is a Cauchy surface if it is a closed
space-like surface homeomorphic to $S$ such that any inextendible time-like
curve on $M$ intersects $\Sigma$ exactly once. The maximality condition 
then says that any isometric embedding
$(M,g)\to (M',g')$, with $(M',g')$ also satisfying the two conditions above, 
is a global isometry.

The identity component $\isom_0(AdS^3)$ of the isometry group of $AdS^3$ is isomorphic to 
$\PSL(2,\R)\times \PSL(2,\R)$. Thus, since the holonomy representation $\rho$ 
of a GHM AdS metric $g$ on 
$M$ has values in $\isom_0(AdS^3)$, it can be decomposed as $\rho=(\rho_L,\rho_R)$,
where $\rho_L,\rho_R$ are morphisms from $\pi_1S$ to $\PSL(2,\R)$, well-defined up to
conjugation. We will call $\rho_L$ and $\rho_R$ the left and right representations of $g$.

The following result by Mess \cite{mess,mess-notes} provides a classification of GHM AdS
manifolds in terms of their holonomy representations and can be considered as an analog
of the Bers Double Uniformization Theorem (see Theorem \ref{tm:bers}).

\begin{theorem}[Mess] \label{tm:mess}
The representations $\rho_L$ and $\rho_R$ have maximal Euler number, so that they
are by \cite{goldman:topological} holonomy representations of hyperbolic 
structures $m_L, m_R\in \cT$. 
Given $(\rho_L,\rho_R)\in \cTT$, there is a unique GHM AdS metric $g\in \cGH_{-1}$
such that $\rho_L$ and $\rho_R$ are the left and right representations of $g$.
\end{theorem}

As a consequence, we have a homeomorphism $hol:\cGH_{-1}\to \cTT$, sending $g$ to 
$(\rho_L,\rho_R)$. Moreover, $\cT$ is equipped with a symplectic structure, given
by the Weil-Petersson symplectic form $\omega_{WP}$, so that $\cTT$
is also equipped with a symplectic form $\omom$. Pulling back $\omom$ by $hol$ gives a real symplectic
structure on $\cGH_{-1}$.

\subsection{Convex cores of quasifuchsian and globally hyperbolic manifolds}
\label{ssc:convexcores}

Now consider a quasifuchsian metric $h$ on $M$. According to the definition given above,
$M$ contains a non-empty, compact, convex subset $K$. It is easily seen that
the intersection of two non-empty convex subsets is also convex, and it follows 
that $M$ contains a unique smallest non-empty convex subset, called its 
convex core and denoted here by $C(M,h)$. 

In some cases, $C(M,h)$ is a totally geodesic surface. This happens exactly when 
$M$ is ``Fuchsian'', that is, the image of its holonomy representation is conjugate to a 
subgroup of $\PSL(2,\R)\subset \PSL(2,\C)$. Otherwise, when $M$ is non-Fuchsian,
$C(M,h)$ has non-empty interior. Its boundary $\partial C(M,h)$ is then the disjoint union
of two surfaces $S_+$ and $S_-$ homeomorphic to $S$, facing 
respectively towards the upper and lower asymptotical boundaries $\partial_+M$ and $\partial_-M$ of $M$.

Both $S_+$ and $S_-$ are locally convex surfaces with no
extreme points. It follows (see \cite{thurston-notes}) that their induced metrics
$m_+$ and $m_-$ are hyperbolic, and that they are pleated along measured laminations
$l_+$ and $l_-$. This associates to $h\in \cQF$ a pair of hyperbolic metrics $m_+,m_-\in \cT$ and 
a pair of measured laminations $l_+, l_-\in \cML$. These data are however not independent with, say,
the pair $(m_-,l_-)$ on the lower boundary of the convex core being completely determined by the pair
$(m_+,l_+)$ on the upper boundary. Thus, restricting our attention to the upper boundary, we obtain a map
$$\partial_+^{Hyp}:\cQF\to\cT\times\cML$$
associating to a quasifuchsian metric $h$ the data $(m_+,l_+)$ on $S_+$.

An analogous description of the convex core applies to GHM AdS manifolds, see \cite{mess,mess-notes}, leading to an analogous map
$$\partial_+^{AdS}:\cGH_{-1}\to\cT\times\cML$$
associating to a GHM AdS metric $g$ a pair $(m_+,l_+)$ on the upper boundary of the convex core $C(M,g)$.

\subsection{Hyperbolic ends and complex projective structures}

The description of the upper boundary of the convex core of quasifuchsian manifolds can be extended as follows.
Consider a quasifuchsian manifold $(M,h)$ homeomorphic to $S\times\R$, 
and let $E_+$ be the upper connected component of $M\setminus C(M,h)$. 
It is a non-complete hyperbolic manifold, homeomorphic to $S\times (0,\infty)$,
which is complete on the side corresponding to $\infty$, and bounded on the
side corresponding to $0$ by a concave pleated surface. A hyperbolic manifold
of this type is called a (non-degenerate) {\it hyperbolic end}. We call 
$\cH\cE$ the space of (non-degenerate) hyperbolic ends homeomorphic to $S\times (0,\infty)$.

Given a hyperbolic end $(E,h)$, we call $\partial_\infty E$ its ``boundary at infinity''
corresponding to the ``complete'' side, and $\partial_0E$ its boundary component
which is a concave pleated surface. The universal cover $\Et$ of $E$ admits a
developing map with values in $\bbH^3$, which restricts to a developing map of 
$\widetilde{\partial_\infty E}$ into $\partial_\infty \bbH^3$, which can be identified
with $\C P^1$. Since hyperbolic isometries act on $\partial_\infty \bbH^3\simeq \C P^1$
as projective transformations, $\partial_\infty E$ is endowed with a complex
projective structure $\sigma\in \cCP$. Conversely, one can associate a hyperbolic
end to any complex projective structure on $S$, so $\cHE$ is in bijection with $\cCP$.

The holonomy representation $\rho:\pi_1S\to PSL(2,C)$ of a hyperbolic end 
can be considered as an element of $\cX(\pi_1S,\PSL(2,\C))$,
and the corresponding map $hol:\cHE\to \cX(\pi_1S,\PSL(2,\C))$ can be used to pull
back on $\cHE$ the symplectic form $\omega_G$. As for quasifuchsian manifolds, we 
will denote by $\omega_G^i$ the imaginary part of this Goldman symplectic form on 
$\cHE$. Note that $\cQF$ has a natural embedding in $\cHE$, sending a quasifuchsian
metric to the upper connected component of the complement of its convex core.
Our notations are compatible with this embedding (in that it sends $\omega_G^i$ on
$\cQF$ to $\omega_G^i$ on $\cHE$).

\subsection{Wick rotations}

The heuristic idea of Wick rotation is old and quite natural. The underlying space-time of
special relativity is the Minkowski space, that is, $\R^4$ with the Lorentzian metric 
$-dt^2+dx^2+dy^2+dz^2$. Mathematicians (and physicists at the time) were used to 
the four-dimensional Euclidean space, $\R^4$ with the bilinear form $d\tau^2+dx^2+dy^2+dz^2$. 
A simple way to pass from one to the other is to ``complexify time'', that is, write
$t=i\tau$, so that the Minkowski metric is written in terms of the variables 
$(\tau,x,y,z)$ exactly as the Euclidean metric. 

The ``Wick rotations'' that we consider here, following \cite{benedetti-bonsante},
are slightly more elaborate versions
of the same idea. We consider a constant curvature metric $g$ on a 3-dimensional
manifold $M$ (homeomorphic to $S\times \R$) along with a surface $\Sigma\subset M$.
(The metric $g$ can be hyperbolic or Lorentzian of curvature $-1,0$ or $1$, and
the surface $\Sigma$ is always ``special'', it can be a minimal or
maximal surface, a CMC surface, or a pleated surface.) We then note that under
various hypothesis there is a unique metric $g'$ on $M$ which is also of constant
curvature, but of a different type than $g$, containing a surface $\Sigma'$
which is either isometric or conformal to $\Sigma$, and ``curved'' in the same way,
in the sense that they have the same second fundamental form (traceless
second fundamental form or measured bending lamination, depending on the case considered).

We are thus interested in the relations between moduli space of geometric structures
on $S\times \R$, in particular
\begin{itemize}
\item quasifuchsian hyperbolic metrics (see Definition \ref{df:quasi}), or more
generally hyperbolic ends,
\item maximal globally hyperbolic anti-de Sitter metrics (see Definition \ref{df:ads}),
\end{itemize}
but also maximal globally hyperbolic de Sitter or Minkowski metrics, to be defined below. We give the 
main definitions first for maps between spaces of quasifuchsian metrics (more generally of hyperbolic ends)
and the space of globally hyperbolic AdS metrics. In each case, the map
is defined on only parts of the space of quasifuchsian metrics.% and/or its image is only a subdomain of the space of GHM AdS metrics.

\subsubsection{Convex pleated surfaces}

Let $(E,h)$ be a hyperbolic end and let $(m_+,l_+)=\partial_+^{Hyp}(h)$
be the induced metric and measured pleating
lamination on $\partial_0E$. The hyperbolic metric $m_+$ lifts to a complete
hyperbolic metric $\tilde m_+$ on the universal cover $\widetilde{\partial_0E}$, 
and $l_+$ lifts to a measured geodesic lamination $\tilde l_+$ for $\tilde m_+$.

The data $(\tilde m_+,\tilde l_+)$ then defines a unique pleated surface $\tilde\Sigma$ in $AdS^3$ 
(see \cite{benedetti-bonsante}) which by construction is invariant and cocompact 
under an action $\rho:\pi_1S\to \isom(AdS^3)$. This action extends in a properly
discontinuous manner to a small tubular neighborhood (the domain of dependence) of $\tilde\Sigma$ in $AdS^3$, and
taking the quotient of this tubular neighborhood by $\rho(\pi_1S)$ defines an
AdS $3$-manifold $(M',g')$ which, by construction, is globally hyperbolic. Therefore
$(M',g')$ embeds isometrically in a unique GHM AdS manifold $(M,g)$ (see \cite{mess}). By
construction, $\Sigma/\rho(\pi_1S)$ embeds isometrically as a pleated surface in 
$M$ which can only be the upper boundary of the convex core of $M$, 
so that $(m_+,l_+)$ is also the data defined on the upper boundary of $C(M,g)$
$$ (m_+,l_+)=\partial_+^{AdS}(g)~. $$

This construction defines a ``Wick rotation'' map $W_{\partial}^{AdS}:\cHE\to\cGH_{-1}$, see also
\cite{benedetti-bonsante}, via
$$W_{\partial}^{AdS}=(\partial_+^{AdS})^{-1}\circ\partial_+^{Hyp}.$$

The following proposition is perhaps not as obvious as it might appear at first
sight. It is close in spirit to \cite[Lemma 1.1]{cp}. Here the smooth structure
considered on $\cHE$ is for instance the one induced by the embedding described
above of $\cHE$ into $\cX(\pi_1S,PSL(2,\C))$.

\begin{prop} \label{pr:C1}
The map  $W_{\partial}^{AdS}:\cHE\to \cGH_{-1}$ is one-to-one and $C^1$-smooth.
\end{prop}

% % \footnote{We still have to add somewhere a proof that it is one-to-one (easy). Also add a sentence here to say where the proof is.}

This proposition thus implies that we can consider the pull back by $W_{\partial}^{AdS}$ of the symplectic
structure on the target space. We then obtain the following theorem, whose proof can be found in Section \ref{sc:double-earthquake}.

\begin{theorem} \label{tm:wick-cc}
The map $W_{\partial}^{AdS}:(\cH\cE,\omega_G^i)\to (\cGH_{-1},\omom)$ is symplectic.
\end{theorem}
% \footnote{Add here sentence on where it is proved.}

\subsubsection{Minimal or maximal surfaces}

Given a quasifuchsian metric $h\in\cQF$ on $M$, it is known (see e.g. \cite{uhlenbeck}) that $M$
contains a closed minimal surface homeomorphic to $S$. However this minimal surface
is in general not unique. There is a specific class of quasifuchsian manifolds 
containing a unique closed, embedded minimal surface: they are those, called almost-Fuchsian, 
which contain a closed, embedded minimal surface with principal curvatures everywhere in $(-1,1)$, see
\cite{uhlenbeck}. We call $\cAF\subset \cQF$ the space of almost-Fuchsian metrics on
$M$, considered up to isotopy. 

Thus, restricting our attention to $h\in \cAF$, let $\Sigma\subset M$ be its unique closed, 
embedded minimal surface and consider its induced metric $I$ and second
fundamental form $\II$. It is well known (see e.g. \cite{minsurf}) that $\II$
is then the real part of a holomorphic quadratic differential $q$ for the complex 
structure defined on $S$ by the conformal class of $I$. So $([I],\II)$ define
a point $(c,q)\in \cQ$. We thus obtain a map
$$ min:\cAF\to\cQ $$
sending an almost-Fuchian metric to the data on its minimal surface.

Things are simpler for GHM AdS manifolds. It is well known (see e.g. \cite{minsurf}) 
that any GHM AdS manifold contains a
unique closed, space-like maximal surface. Moreover, given a complex structure $c$ and a holomorphic quadratic
differential $q$ for $c$ on $S$, there is a unique GHM AdS metric $h$ on $M$ such
that the induced metric and second fundamental form on the unique maximal
surface in $M$ is $I,\II$ with $I$ compatible with $c$ and $\II=Re(q)$.

This provides an analogous map
$$max:\cGH_{-1}\to\cQ$$
sending an GHM AdS metric to the data on its maximal surface, which by the arguments 
above is one-to-one.

Now, given $h\in \cAF$, we can consider the induced metric $I$ and second fundamental
form $\II$ on its unique closed, embedded minimal surface, and the consider the unique GHM AdS
metric $h$ on $M$ containing a maximal surface with induced metric conformal to $I$
and second fundamental form equal to $\II$. This defines another ``Wick rotation''
map $W_{min}:\cAF\to \cGH_{-1}$ via
$$W_{min}:max^{-1}\circ min.$$
This map is smooth, and we have the following result.

\begin{theorem} \label{tm:minimal}
The  map $W_{min}:(\cAF, \omega_G^i)\to (\cGH_{-1},\omom)$ is symplectic.
\end{theorem}

\subsubsection{Constant mean curvature surfaces}

The following picture can be extended by considering constant mean curvature (CMC)
surfaces, rather than minimal or maximal surfaces. Recall that the mean curvature
of a surface in a Riemannian or Lorentzian $3$-manifold is the trace of its shape
operator. We will use a basic and well-known fact (see \cite{hopf}): the traceless part
of the second fundamental form of an oriented constant mean curvature surface in any constant
curvature $3$-dimensional (Riemannian or Lorentzian) manifold is the real part of
a holomorphic quadratic differential, for the complex structure associated to
its induced metric.

GHM AdS manifold are particularly well-behaved with respect to CMC surfaces.
On one hand, any GHM AdS manifold contains a canonical foliation by CMC surfaces.

\begin{theorem}[Barbot, B\'eguin, Zeghib \cite{BBZ}] \label{tm:ads:foliation}
Any GHM AdS manifold $M$ admits a unique foliation by closed space-like CMC surfaces,
with mean curvature varying between $-\infty$ and $\infty$. For all $H\in \R$, $M$
contains a unique closed space-like CMC-$H$ surface. 
\end{theorem}

On the other hand, one can also associate through CMC-$H$ surfaces a GHM AdS manifold
to any point in $T^*\cT$, thanks to the following proposition (see \cite[Lemma 3.10]{minsurf}).

\begin{prop}
Let $H\in (-\infty,\infty)$.
Given a complex structure $c$ and a holomorphic quadratic
differential $q$ for $c$ on $S$, there is a unique GHM AdS metric $h$ on $M$ such
that the induced metric and traceless part of the second fundamental form on the unique CMC-$H$
surface in $(M,h)$ is $I,\II_0$ with $I$ compatible with $c$ and $\II_0=Re(q)$.
\end{prop}

In the quasifuchsian context, it was conjectured by Thurston that the 
analog of Theorem \ref{tm:ads:foliation} also holds, but
only for almost-Fuchsian manifolds. Lacking a proof of this fact, we introduce the
following notation.

\begin{defi}
We denote by $\cAF'$ the space of quasifuchsian metrics on $M$ which admit a unique
foliation by CMC surfaces with $H\in (-2,2)$.
\end{defi}

Note that the Thurston conjecture mentioned above can be reformulated as the fact that $\cAF\subset \cAF'$, as any closed embedded CMC surface is a leaf of the foliation by the maximum principle.

% \begin{theorem}[\cite{}]
% Any almost-Fuchsian manifold $M$ has a unique by closed CMC surfaces, with constant 
% mean curvature varying between $-2$ and $2$. Moreover for all $H\in (-2,2)$, 
% $M$ contains a unique closed CMC-$H$ surface.
% \end{theorem}

We can use now construct a generalization of the map $W_{min}$ associated
to a pair of constant mean curvatures. Let $H\in (-2,2)$ and $H'\in (-\infty,\infty)$. For each $h\in \cAF'$, let $S_H$ be the
unique closed CMC-$H$ surface in $(M,h)$, let $c$ be the conformal class of its 
induced metric, and let $q$ be the traceless part of its second fundamental form. 
There is then a unique GHM AdS metric $g$ on $M$ such that the (unique) CMC-$H'$ surface
in $(M,g)$ has induced metric conformal to $c$ and the traceless part of its second fundamental
form is equal to $q$. We denote by $W^{AdS}_{H,H'}:\cAF'\to \cGH_{-1}$ the map sending $h$ to $g$.

\begin{theorem} \label{tm:wick-cmc}
For all $H\in (-2,2)$ and $H'\in (-\infty,\infty)$, 
the map $W_{H,H'}^{AdS}:(\cAF', \omega_G^i)\to (\cGH_{-1},\omom)$ is symplectic.
\end{theorem}

\subsection{Harmonic maps}

We now translate the above stated results purely in terms of surfaces, and of maps between
moduli spaces of surfaces, with no reference to 3-dimensional manifolds. We consider
first harmonic maps, and then earthquakes.

Recall that a map $\phi:(M,g)\to (N,h)$ between two Riemannian manifolds is harmonic if
it is a critical point of the Dirichlet energy, defined as 
$$ E(\phi)=\int_M \|d\phi\|^2 dvol~. $$
If $M$ is a surface, the Dirichlet energy is invariant under conformal deformations
of the metric on $M$, so the notion of harmonic maps from $M$ to $N$ only depends
on the choice of a conformal class (no Riemannian metric is needed).

Let's now consider the case of harmonic self-maps of a surface $S$. Let $c\in\cT$ be a complex structure on $S$
and $m\in\cT$ a hyperbolic metric. Given a map $\phi:(S,c)\to (S,m)$, its Hopf differential $\hopf(\phi)$ is defined as
the $(2,0)$ part of $\phi^*m$. 
A key relation between holomorphic quadratic differentials and harmonic maps is that $\hopf(\phi)$ is holomorphic if 
and only if $\phi$ is harmonic. In addition, we will use the following well-know
statements. 

\begin{theorem}[Eells and Sampson \cite{eells-sampson:harmonic}, Hartman \cite{hartman:harmonic},
Schoen and Yau\cite{schoen-yau:78}] 
\label{tm:existence-harmonic}
If $S$ is a closed surface equipped with a conformal class $c$ and $m$ is any
hyperbolic metric on $S$, then there is a a unique harmonic map isotopic to the
identity from $(S,c)$ to $(S,m)$.
\end{theorem}

\begin{theorem}[Sampson \cite{sampson}, Wolf \cite{wolf:teichmuller}] \label{tm:hopf}
Let $c\in \cT$ be a complex structure on a surface $S$, and let $q$ be a holomorphic 
quadratic differential on $(S,c)$. There is a unique hyperbolic metric $m$ on $S$,
well-defined up to isotopy, 
such that the Hopf differential of the harmonic map $\phi:(S,c)\to (S,m)$ isotopic
to the identity is equal to $q$.
\end{theorem}

Together these define a map $H:\cQ\to\cT$, sending $(c,q)$ to $m$.

\begin{defi}
We denote by $\cH:\cQ\to \cT\times \cT$ the map defined by 
$$ \cH(c,q) = (H(c,-iq),H(c,i q))~. $$
We will call $\cH$ the double harmonic map.
\end{defi}

It is a well known fact that the bundle $\cQ$ of holomorphic quadratic 
differentials can be identified with the holomorphic cotangent bundle 
$T^{*(1,0)}\cT$ over Teichm\"uller space. We denote by $\omega_*$ the canonical 
complex cotangent bundle symplectic structure on $T^{*(1,0)}\cT$ and by 
$\omega_*^r$ its real part, which corresponds to the real canonical cotangent 
bundle symplectic structure on $T^*\cT$. We then obtain the following result.

\begin{theorem} \label{tm:harmonic}
$\cH:(\cQ,\omega_*^r)\to (\cT\times \cT,\omom)$ is symplectic, up to a 
multiplicative factor:
$$\cH^*\omom=-\omega_*^r$$
\end{theorem}

\subsection{Earthquakes}

Turning now to a more geometric construction, we now consider the bundle of measured
geodesic laminations over $\cT$.

\begin{defi}
A measured geodesic lamination is a closed subset $l\subset S$ which is foliated
by complete simple geodesics, defined with respect to a given hyperbolic metric
$m\in\cT$, together with a positive measure $\mu$ on arcs transverse to the
leafs of $l$ which is invariant under deformations among transverse arcs with fixed 
endpoints (see e.g. \cite{bonahon-geodesic}). We denote the space of measured
geodesic laminations on $S$, considered up to isotopy, by $\cML$.
\end{defi}

Similarly to holomorphic quadratic differentials, the definition of measured
geodesic laminations depends on the choice of a point in $\cT$ and thus
determines a bundle over Teichm\"uller space. However, unlike $\cQ$, there is a
canonical identification between the fibres over any pair $m,m'\in\cT$ --- this
extends the fact that any
simple closed geodesic for $m'$ is isotopic to a unique simple closed geodesic
for $m$ (see \cite{bonahon-geodesic}).
This justifies the notation of $\cML$ without any reference to a
hyperbolic structure.

Given a hyperbolic metric $m\in\cT$ and a measured geodesic lamination
$l\in\cML$ we may define the left earthquake of $m$ along $l$. This is a new
hyperbolic metric on $S$ denoted by $E_L(m,l)$. For $l$
supported on a simple close geodesic $\gamma$ with weight $a$, $E_L(m,l)$ is
defined by cuting $S$ along $\gamma$, rotating the left-hand side of $\gamma$ by
length $a$ and then gluing it back. The operation for general laminations is then defined as
certain (well-defined) limiting procedure \cite{thurston-earthquakes}. Importantly we have the
following result, which is a geometric analogue to the analytic results above.

\begin{theorem}[Thurston \cite{thurston-earthquakes}]
For any pair $m,m'\in\cT$ of hyperbolic metrics on $S$ there exists a unique
measured lamination $l\in\cML$ such that $m$ and $m'$ are related by a left
earthquake $m'=E_L(m,l)$.
\end{theorem}

More specifically, we have:

\begin{theorem}[Thurston \cite{thurston-earthquakes}, Kerckhoff \cite{kerckhoff:analytic}]
The map $E_L:\cT\times\cML\to\cT$ is a homeomorphism for every fixed $m\in\cT$
and a real analytic diffeomorphism for every fixed $l\in\cML$.
\end{theorem}

The notion of right earthquake is obtained in the same way, by rotating
in the other direction, so that the right earthquake along $l$, 
$E_R(l)$, is the inverse of $E_L(l)$.
So we have two maps $E_L, E_R:\cT\times \cML\to \cT$. 

\begin{defi} \label{df:cE}
We denote by $\cE:\cT\times \cML\to \cTT$ the map defined by 
$$ \cE(m,l) = (E_L(m,l),E_R(m,l))~. $$
We will call $\cE$ the double earthquake map.
\end{defi}

Note that $\cE$ is a bijection. Indeed, from Thurston's Earthquake
Theorem, given any
pair $m,m'\in \cT$, there is a unique left earthquake path going from $m'$
to $m$. In other terms, there is a unique measured lamination $l\in \cML$
such that $E_L(l)(m')=h$. Now let $m''=E_L(l/2)(m')$. Then clearly
$(m,m')=\cE(m'',l/2)$. Conversely, given any $(m'',l)\in \cT\times \cML$
such that $(m,m')=\cE(m'',l)$, $m''$ must be the midpoint of the left earthquake
path from $m'$ to $m$, and this path is associated to $l$, so the map $\cE$ is
one-to-one.

However there is no reason to believe that $\cE$ is differentiable --- actually it is not even
clear what it would mean, since there is no canonical differentiable structure on 
$\cML$. To deal with this differentiability issue we introduce a map $\delta:\cT\times 
\cML\to T^*\cT$ which sends a hyperbolic metric $m\in 
\cT$ on $S$
and a measured lamination $l\in \cML$ to the differential at $m$ of the length function of $l$,
$L(l):\cT\to \R$,
$$ \delta(m)=d_mL(l)~. $$
This map is a global homeomorphism between $\cT\times \cML$ and  
$T^*\cT$, see 
\cite[Lemma 2.3]{cp}.

The following can be seen as a translation of Proposition 
\ref{pr:C1}, see Section \ref{sc:reg_earthquake} for a proof.

\begin{prop} \label{pr:EC1}
The map $E_l\circ \delta^{-1}:T^*\cT\to \cT$ is $C^1$-smooth.
\end{prop}

\begin{cor} \label{cr:EC1}
  $\cE\circ \delta^{-1}:T^*\cT\to \cTT$ is a $C^1$ diffeomorphism.   
\end{cor}

This corollary then allows to consider the following 
statement, whose proof can be found in Section \ref{sc:double-earthquake}

\begin{theorem} \label{tm:double-earthquake}
The map $\cE\circ\delta^{-1}:(T^*\cT, \omega_*^r)\to (\cT\times \cT, \omom)$
is symplectic, up to a multiplicative factor:
$$ (\cE\circ\delta^{-1})^*\omom = 2\omega_*^r~. $$
\end{theorem}

Here $\omega_*^r$ is the real part of the cotangent symplectic structure on $T^*\cT$.

\subsection{Generalizations for Minkowski and de Sitter manifolds}

Globally hyperbolic maximal Minkowski and de Sitter manifolds are defined in a similar way to GHM AdS manifolds.
We denote by $\cGH_0$ and $\cGH_1$ the corresponding moduli spaces of Minkowski and de Sitter metrics on 
$M=S\times\R$, respectively. Their classification are obtained in \cite{mess,scannell:de_sitter} and can
be summarised as follows.

\begin{theorem}[Mess]
The holonomy representation of a GHM Minkowski metric $g$ on $M$ decomposes as $\rho=(\rho_0,\tau)$, where $\rho_0:\pi_1S\to PSL(2,\R)$
is a Fuchian representation and $\tau:\pi_1S\to\psl(2,\R)$ is an element of the first cohomology group $H^1(\pi_1S,\psl(2,\R)_{\Ad \rho_0})$.
Thus, by the identification of $H^1(\pi_1S,\psl(2,\R)_{\Ad \rho_0})$ with the fiber over $\rho_0$ of the cotangent bundle over the $PSL(2,\R)$
representation variety, the holonomy map provides a homeomorphism $hol:\cGH_0\to T^*\cT$.
\end{theorem}

\begin{theorem}[Scannell]
The developing map $dev:\tilde M\to dS^3$ of a GHM dS metric $g$ on $M$ extends to a developing map 
$dev:\widetilde{\partial_+ M}\to \partial_+ dS^3\simeq \C P^1$. The holonomy representation 
$\rho:\pi_1S\to PSL(2,\C)$ endows $\partial_+M$ with a complex projective structure. This 
defines a map $\partial_\infty^{dS}:\cGH_1\to\cCP$ which is a homeomorphism.
\end{theorem}

It is then also possible to define similar Wick-rotations maps $W^{Mink}_\partial:\cHE\to\cGH_0$,
$W^{Mink}_{H,H'}:\cAF'\to\cGH_0$ and $W^{dS}_\partial:\cHE\to\cGH_1$
and $W^{dS}_{H,H'}:\cAF'\to\cGH_1$. The main difference is that these manifolds now do
not contain convex pleated surfaces nor maximal surfaces. 
Their relation with hyperbolic manifolds in terms of measured laminations is still possible 
via: (1) the inital singularity of flat manifolds and (2) the projective duality between 
hyperbolic and de Sitter manifolds. The relation in terms of CMC foliations is also available in 
both cases, only with $H'$ varying between $(-\infty,0)$ and $(-\infty,-2)$, respectively.

We shall prove, in the de Sitter case, that the CMC Wick rotations $W^{dS}_{H,H'}$ is again symplectic,
where the symplectic structure on $\cGH_1$ is again the pull-back of the imaginary part of the Goldman 
symplectic structure on $\cCP$.

% In the de Sitter case, we can define a CMC Wick rotations $W^{dS}_{H,H'}$ between
% the moduli space $\cAF'$ of almost-Fuchsian hyperbolic manifolds and the moduli space
% $\cGH_1$ of globally hyperbolic de Sitter manifolds as follows.
% % \begin{defi}
% Let $H\in (-2,2), H'\in (-\infty, -2)$. Any almost-Fuchian $M\in \cAF'$ contains a
% unique CMC-$H$ surface $\Sigma$. Let $c=[I]$ be the conformal class of the induced
% metric on $\Sigma$, and let $q$ be the holomorphic quadratic differential such
% that the traceless part $\II_0$ of the second fundamental form of $\Sigma$ is the
% real part of $q$. There is then by \cite[Lemma 6.1]{minsurf} a unique GHM de Sitter
% manifold $W_{H,H'}^{dS}$ containinig a CMC-$H'$ surface with induced metric conformal
% to $c$ and traceless second fundamental form equal to $Re(q)$.
% % \end{defi}

\begin{theorem} \label{tm:wick-ds}
Let $H\in (-2,2)$ and let $H'\in (-\infty, -2)$. The map 
$W^{dS}_{H,H'}:(\cAF',\omega_G^i)\to (\cGH_1,\omega_G^i)$ is symplectic. 
\end{theorem}

The proof can be found in Section \ref{ssc:dS-CMC-wick}.

%% parameterization of $\cGH_1$ by $\cCP$, Scannell. $\omega_G^i$ in this way.

%%  parameterizations by $T*T$ by CMC surfaces. 
%% same symplectic structure as from infinity (?)

%%% Added the following short subsection on perspectives
\subsection{Spaces with particles}

The results above might have extensions to constant curvature 3-manifolds of various types
containing ``particles'', that is, cone singularities of angle less than $\pi$ along infinite
geodesics connecting the two connected components of the boundary at infinity (in a ``quasifuchsian''
hyperbolic manifold) or along a maximal time-like geodesic (in a GHM AdS, dS or Minkowski 
spacetime). 

A number of the tools needed to state the results above are known to extend to this setting.
For quasifuchsian manifolds, an extension of the Bers double uniformization theorem is known in this setting 
\cite{qfmp,conebend}. The Mess analog for GHM AdS manifolds also extends to this setting with ``particles''
\cite{cone}, and the existence and uniqueness of a maximal surface (orthogonal to the particles) also
holds \cite{toulisse:maximal}. However it is not known whether GHM AdS, dS or Minkowski space-times 
with particles contain a unique foliation by CMC surfaces orthogonal to the particles.

\subsection{Physical motivations}

From a physical point of view there are two approaches to understand the relation between Teichm\"uller theory and 3d gravity which motivates the existence of the symplectic maps considered in the present work. In both these descriptions one rewrites the action principle determined by the Einstein-Hilbert functional in terms of new variables as to simplify the description of the moduli space of critical points.

We remind the reader that the Einstein-Hilbert functional on the space of 3-dimensional Lorentzian metrics on the spacetime manifold $M=S\times\R$ is defined by
\begin{equation*}
S[g]=\frac{1}{2}\int_M(R-2\Lambda)dv+\int_{\partial M}H da,
\end{equation*}
where $dv$ and $R$ are the volume form and the scalar curvature of $M$, $\Lambda=0,-1,1$ the cosmological constant, and $da$ and $H$ the area form and the mean curvature of $\partial M$. The critical points are given by solutions of Einstein's equation
\begin{equation*}
\Ric-\frac{1}{2}(R-2\Lambda)g=0
\end{equation*}
and the corresponding moduli space is the classical phase space of the theory, which is naturally endowed with a symplectic structure defined from the action principle.
%Note that the Einstein-Hilbert functional may be digergent, even for compact spatial topologies $S$, due to the non-compacity of the temporal dimension. Thus, it may be needed to consider a renormalized version of $S[h]$ obtained by subtractig an appropriate counter term.
%% should we add a sentence about renormalization when this diverges ?

The first approach to describe the moduli space of critical points of the Einstein-Hilbert action follows from the interpretation of Einstein's equation as a constrained dynamical system for 2-dimensional Riemannian metrics on $S$, see \cite{ADM,Moncrief}. One starts with a choice of constant time foliation of the spacetime $M$ and a decomposition of the 3-dimensional metric into spatial and temporal components. The Einstein-Hilbert functional can then be rewritten in terms of the induced metric $I$, the extrincic curvature $\II$ of the leaves of the foliation and three Lagrange multipliers imposing constraints on $I$ and $\II$
$$S[I,\II]=\frac{1}{2}\int_\R dt\int_S da \langle\II-2H I,\dot I\rangle-\lambda\cC.$$

The constraints $\cC=0$ are nothing but the Gauss-Codazzi equations relating the first and second fundamental forms of an embedded surface to the ambient geometry of the spacetime. In terms of the conformal structure $c$ determined by $I$ we may write
$$I=e^{2\varphi}c,\qquad \II=\re(q)+e^{2\varphi}Hc,$$
and the Gauss-Codazzi equation becomes
$$4\partial_z\partial_{\bar z}\varphi=e^{2\varphi}(H^2-\Lambda)-e^{-2\varphi}|q|^2,\qquad \partial_{\bar z}q=e^{2\varphi}\partial_{\bar z}H.$$

For maximal globally hyperbolic spacetimes, the equations of motion are then uniquely solved given intial data on a Cauchy surface $\Sigma$. Also in this case it is always possible to choose a foliation containing a constant mean curvature ($H=const.$) initial Cauchy surface. The constraints are then easily solved: the Codazzi constraint equation becomes a holomorphicity equation for the quadratic differential $q=qdz^2$ determined by the traceless part of $\II$ and the Gauss constraint equation becomes an elliptic differential equation for $e^{2\varphi}$, the conformal factor of $I$. The existence and uniqueness of solutions of the Gauss equation are guaranteed for $H^2-\Lambda\geq1$, see \cite{Moncrief}, thus showing that the inital data parameterizing the moduli space of globally hyperbolic maximal spacetimes is given by points in the cotangent bundle over Teichm\"uller space of the initial Cauchy surface:
$$\cGH_\Lambda=T^*\cT~. $$

The symplectic structure on $\cGH_\Lambda$ is then further shown to agree, up to a multiplicative constant, with the real canonical symplectic structure $\omega_*^r$ on $T^*\cT$. This follows via symplectic reduction of $T^*\Riem(S)$, the cotangent bundle over Riemannian metrics on $S$ with its canonical symplectic structure, to the constraint submanifold defined by the Gauss-Codazzi equation \cite{Moncrief}.

The second approach to describe the moduli space $\cGH_\Lambda$ stems from the fact that 3d Einstein manifolds have constant sectional curvature, equal to the cosmological constant $\Lambda$. Therefore, such manifolds can be described as quotients of appropriate domains of either Minkowski, anti-de Sitter or de Sitter 3-spacetime, in the Lorentzian setting, and Euclidean, hyperbolic or spherical 3-space, in the Riemannian setting. The study of 3d Einstein manifolds can thus be viewed in the context of locally homogeneous geometric structures, i.e. flat $G$-bundles over spacetime. Such an approach was first suggested in the physics literature in \cite{Achucarro,Witten} where the Einstein-Hilbert action is shown to be equivalent to a Chern-Simons action on the space of $G$-connections over the spacetime manifold. Here $G$ is the isometry group of the relevant model spacetime, that is, $\PSL(2,\R)\ltimes\psl(2,\R)$ for $\Lambda=0$, $\PSL(2,\R)\times\PSL(2,\R)$ for $\Lambda=-1$, and $\PSL(2,\C)$ for $\Lambda=1$.

This is obtained by first decomposing the spacetime metric $g$ in terms of a coframe field $e$ and spin connection $\omega$, which are taken to be independent. By appropriately tensoring the components of $e$ and $\omega$ with Lie algebra generators one then constructs the associated $\mathfrak{g}$-valued 1-form $A$ on $M$. Finally, translating the Einstein-Hibert action for $g$ in terms of $A$ gives exactly the Chern-Simons action
$$S_{G}[A]=\frac{1}{2}\int_M\tr(A\wedge dA+\frac{2}{3}A\wedge A\wedge A)~.$$
This provides a description of the moduli space of spacetimes as a subspace of the moduli space of flat $G$-connections on $S$. In the maximal globally hyperbolic case it is possible to describe the gravitational component completely \cite{mess,scannell:de_sitter}
$$\cGH_{\Lambda}=\begin{cases}T^*\cT & \quad \Lambda=0, \cr \cT\times\overline{\cT} & \quad \Lambda=-1, 
\cr \cCP & \quad \Lambda=1. \end{cases}$$

To describe the symplectic structure in this formulation, recall that the Chern-Simons functional on the space of $G$-connections is defined in terms of an $\Ad$-invariant symmetric bilinear form $\tr_\Lambda$ on $\mathfrak{g}$, which must be non-degenerate. The symplectic form then depends essentially on $\tr_\Lambda$, given by the Goldman cup product symplectic form with coefficient pairing given by $\tr_\Lambda$. For the isometry groups of the 3d geometric models described above, the corresponding Lie algebras are know to admit a (real) 2-dimensional space of such bilinear forms. Thus, there is a 2-dimensional family of (real) symplectic forms on the corresponding moduli spaces. In \cite{Witten} Witten obtained the relevant bilinear forms for gravity, that is, the ones arrising from the Einstein-Hilbert functional. These are given by
$$\tr_\Lambda(X,Y)=\begin{cases} \tr_0((x,u),(y,v))=\tr(xv)+\tr(yu) & \quad \Lambda=0, \cr \tr_{-1}((x_+,x_-),(y_+,y_-))=\frac{1}{2}\tr(x_+y_+)-\frac{1}{2}\tr(x_-y_-)  & \quad \Lambda=-1, \cr \tr_1(z,w)=\mathrm{Im}\,\tr(zw) & \quad \Lambda=1.\end{cases}$$
This identifies the relevant symplectic forms on the moduli spaces $\cGH_\Lambda$: for $\Lambda=0$ the symplectic form is given by $\omega_*^r$, the real canonical cotangent bundle symplectic form on $T^*\cT$, for $\Lambda=-1$ it is given by $\omega_{WP}\oplus\overline{\omega_{WP}}$, the difference of Weil-Petersson symplectic forms on each copy of $\cT$, and for $\Lambda=1$ by $\omega_G^i$, the imaginary part of the (complex) Goldman symplectic form on $\cCP$.

\subsection{Content of the paper}

Section \ref{sc:background} contains background material on various aspects of 
the geometry of surfaces and 3-dimensional manifolds, which are necessary elsewhere,
including the definitions and basic properties of quasifuchsian manifolds and of
globally hyperbolic spacetimes of various curvatures, statements on maximal and
CMC surfaces, convex cores, as well as measured laminations and transverse cocycles.

In Section \ref{sc:wick} a more complete description of the double harmonic and double earthquake
map, as well as of the Wick rotation map. We describe the precise relation between those
``double'' maps and the Wick rotation maps, and show the equivalence between statements on
the ``double'' maps and statements on the Wick rotation maps. We prove that the double
earthquake and double harmonic map are one-to-one.

Section \ref{sc:reg_earthquake} is mostly focused on the regularity of the double
earthquake map, and therefore of the earthquake map itself. Section \ref{sc:double-earthquake}
contains the proof that the double earthquake map is symplectic, and then that the double
harmonic map is symplectic --- the connection between the two statements uses a volume
argument that is developed in Section \ref{ssc:dual-schlafli}.

Section \ref{sc:cmc} is focused on CMC surfaces, while the content of Section \ref{sc:mink-ds}
is centered on Minkowski and de Sitter manifolds.

%%% Local Variables:
%%% TeX-master: "doublemaps"
%%% End:

%% file: doublemaps2.tex
\section{Background material} \label{sc:background}

This section contains a number of definition and known results that are useful or
necessary in the sequel. 

\subsection{Teichm\"uller space}

Let $S$ be a closed oriented surface of genus $g\geq 2$. We shall consider here
two equivalent descriptions of the Teichm\"uller space $\cT$ of $S$.

\begin{defi}
A complex structure $c$ on $S$ is an atlas of $\C$-valued coordinate charts,
whose transition functions are biholomorphic. The Teichm\"uller space $\cT$ can
be defined as the space of all complex structure on $S$ compatible with the
orientation, considered up to isotopy.
\end{defi}

%%%\begin{defi}
A hyperbolic metric on $S$ is Riemannian metric $m$ of negative
constant curvature $-1$. The Teichm\"uller space $\cT$ can be equivalently
defined as the space of all hyperbolic metrics on $S$, again considered up to
isotopy.
%%%\end{defi}

The relation between the two definitions is given through the Riemann
uniformization theorem, which also identifies $\cT$ with a connected component
of the representation variety
$\Rep(S,\PSL(2,\R))=\Hom(\pi_1S,\PSL(2,\R))/\PSL(2,\R)$, associating to each
point in Teichm\"uller space its holonomy representation
$\rho:\pi_1S\to\PSL(2,\R)$. Such holonomy representations of hyperbolic surfaces
are Fuchsian representations, characterized by the maximality of their Euler
number \cite{goldman:topological}.

\subsection{The Weil-Petersson symplectic structure}

The $L^2$-norm 
$$ \Vert q\Vert^2=\int_S \rho_m^{-2}|q|^2 $$
on the bundle $\cQ$ of holomorphic quadratic differentials
induces a hermitian metric on $\cT$ via the well-known identification between $\cQ$ and the 
holomorphic cotangent bundle over Teichm\"uller $T^{*(1,0)}\cT$ space. The 
imaginary part of this hermitian metric is then a symplectic form on $\cT$. We 
denote this symplectic form, which is called the Weil-Petersson 
symplectic form, by $\omega_{WP}$.

It is also possible to describe the Weil-Petersson symplectic structure via 
the trivial bundle $\cT\times\cML$ of measured laminations. In fact, 
$\cT\times\cML$ can be identified with the
real cotangent bundle $T^*\cT$ over Teichm\"uller space via the notion of 
length of measured laminations. For a
simple closed curve $\gamma$ with a weight $a$, this is defined simply as the
product of the weight and the geodesic length
of $\gamma$ with respect to $m\in\cT$
$$L_m(\gamma,a)=aL_m(\gamma).$$
Thurston has shown that this notion extends by continuity to the whole of $\cML$, 
thus defining a function
$L:\cT\times\cML\to\R_+$, $(m,l)\mapsto L(m,l)=L_m(l)$, see \cite{bonahon-geodesic}.

Fixing a lamination
$l\in\cML$ and considering the derivative with respect to the first argument we
get a linear map $d_mL(l):T_m\cT\to\R$, which is a point in $T_m^*\cT$. 
%From results by Wolpert \cite{}, such cotangent vectors can be seen as dual with
%respect to the Weil-Petersson symplectic sturcture to infinitesimal earthquake
%vectors, which furthermore cover $T_m\cT$. 
% Note (JM): I've remove the sentence above since it's unrelated to what we 
% discuss at this point.
This provides a homeomorphism
$\delta:\cT\times\cML\to T^*\cT$. The relation between the cotangent bundle and
the Weil-Petersson symplectic structures in this language is then given in terms of
the Thurston intersection form, see Section \ref{sc:double-earthquake} below.

Lastly, the Weil-Petersson symplectic structure also agrees with the 
restriction of the
Goldman symplectic structure on $\Rep(\pi_1S,\PSL(2,\R))$, defined via the cup
product of cohomology classes with cooeficients paired with the Killing form of
$\psl(2,\R)$ \cite{goldman}.

\subsection{Complex projective structures}

We now consider another type of structure on the surface $S$ which has many parallels
with our previous considerations.

\begin{defi}
A complex projective structure $\sigma$ on $S$ is an atlas of $\C P^1$-valued
coordinate charts, whose transition functions are complex projective
transformations. We denote by $\cCP$ the space of all complex projective
structures on $S$, considered up to isotopy.
\end{defi}

Note that there is a natural projection $p:\cCP\to \cT$ associating to a complex
projective structure $\sigma$ on $S$ its underlying complex structure $c$. The
space $\cCP$ can thus be considered as the total space of a bundle over $\cT$.
There are two possible descriptions of $\cCP$ obtained by analytic or
geometric deformations of a fixed complex projective structure. The first is
related to the bundle $\cQ$ of holomorphic quadratic differentials via the
Schwarzian derivative, while
the second is related to the trivial bundle $\cT\times\cML$ via the operation of grafting along
measured laminations. 

\subsubsection{Measured laminations and grafting} \label{ssc:grafting}

Given a hyperbolic metric $m\in\cT$ and a measured geodesic lamination
$l\in\cML$ one may define a complex projective structure via grafting of $m$ along
$l$ as follows. For $l$ supported on a simple close geodesic $\gamma$ with
weight $a$, $Gr(m,l)$ is defined by cuting $S$ along $\gamma$ and gluing it back
through an euclidian cylinder $\gamma\times[0,a]$. This defines a complex
projective structure on $S$ by complementing the Fuchsian projective structure
of $m$ by the projective structure on $\gamma\times[0,a]$ defined by its natural
embedding as an annulus in $\C^*$, see e.g. \cite{dumas-survey}. As for
earthquakes, the operation of grafting is defined for general laminations via
a limiting procedure. 

\begin{theorem}[Thurston, see \cite{kamishima-tan}]
The map $Gr:\cT\times\cML\to\cCP$ is a homeomorphism.
\end{theorem}

\subsubsection{Quadratic differentials and the Schwarzian derivative}

Given two complex projective structures $\sigma,\sigma'\in \cCP$ 
with the same underlying complex structure $c\in\cT$, the Schwarzian derivative
of the
identity map between $(S,\sigma)$ and $(S,\sigma')$ is a holomorphic quadratic 
differential $\cS(\sigma,\sigma')\in \cQ_c$. The composition rule satisfied by
the Schwarzian derivative means that if $\sigma,\sigma'$ and $\sigma''$ are
three complex projective structures with underlying complex structure $c$, then
$\cS(\sigma,\sigma'')=\cS(\sigma,\sigma')+\cS(\sigma',\sigma'')$. 
This identifies $\cCP$ with the affine bundle of holomorphic quadratic differentials
on $\cT$ (see \cite[\S 3]{dumas-survey})
and we may thus write $\sigma'-\sigma\in\cQ_c$ instead of $\cS(\sigma,\sigma')$.

Note however that the identification $\cCP\simeq\cQ$ depends on the choice of a global
section $\cT\to\cCP$, and there are distinct
``natural'' possible choices for such a section, which induce distinct structures
on $\cCP$. For now, let's consider the natural Fuchsian section given by the
Fuchsian uniformization of Riemann surfaces. Thus, given a complex structure $c$
on $S$, the Riemann Uniformization 
Theorem provides a unique Fuchsian complex projective structure $\sigma_c$
uniformizing $c$. Using this canonical section we can define an identification
$\cS_F:\cCP\to\cQ$, 
sending a complex projective structure $\sigma \in \cCP$ with
underlying complex structure $c=p(\sigma)$ to $(c,\sigma-\sigma_c)\in\cQ$. 
(The subscript ``$F$'' here reminds us that we make use of Fuchsian sections.)

% \begin{theorem}[... \cite{?}]
% The map $\cS_F:\cCP\to\cQ$ is a ... diffeomorphism.\footnote{What is the correct statement here? Needs a ref.}
% \end{theorem}

Another ``natural'' global section $\cT\to\cCP$ will be described bellow making use of Bers'
Double Uniformization Theorem for quasifuchsian manifolds.

\subsection{Quasifuchsian hyperbolic manifolds and hyperbolic ends}

The moduli space $\cHE$ of  hyperbolic ends
is in one-to-one correspondence with both $\cT\times\cML$ and $\cCP$ through
the folloiwing result.

\begin{theorem}[Thurston] \label{tm:grafting}
Given a pair $(m,l)\in\cT\times\cML$ there is a unique non-degenerate hyperbolic
end $E$ such that $\partial_0E$ has induced metric given by $m$ and bending
lamination given by $l$. Also, each $\sigma\in\cCP$ is the complex projective
structure at $\partial_{\infty}E$ of a unique (non-degenerate) hyperbolic end $E$.
The relation between the complex projective structure $\sigma$ and the pair $(m,l)$
is given by the grafting map $Gr:\cT\times\cML\to\cCP$ which furthermore is a homeomorphism.
\end{theorem}

% Using this proposition and somewhat abusing notations, 
% we will denote by $\cCP$ the space of hyperbolic ends on
% $S\times\R_+$, considered up to isotopy, and will implictly identify
% a hyperbolic end with its complex projective structure at infinity.
% \footnote{(JM) Note that I've suppressed the notation $\cHE$ and replaced it by
% $\cCP$. I thought that it simplified things a bit, but perhaps it's not a 
% good idea?!}

% \subsubsection{Symplectic structures}

Via the identification of the holomorphic cotangent bundle $T^{*(1,0)}\cT$ with
the bundle of holomorphic quadratic differentials $\cQ$, we may use the Schwarzian
parametrization $\cS_F:\cCP\to\cQ$ to pull-back the canonical complex 
symplectic structure
$\omega_*$ on $T^{*(1,0)}\cT$ to a complex symplectic structure 
$\omega_F=\cS_F^*\omega_*$
on $\cCP$. We will be interested here only in the real part of 
$\omega_*$, corresponding
to the real symplectic structure on $T^*\cT$. We denote by $\omega_F^r$ the real 
part of $\omega_F$, which is just $\cS_F^*\omega_*^r$

It is important to stress that the Schwarzian parametrization $\cS_F$, and therefore also the symplectic structure $\omega_F$,
depends on the choice of a global section $\sigma_F:\cT\to\cCP$, here taken as the Fuchsian section, which is by no means unique.
In fact, quasifuchsian manifolds provide another class of ``natural'' sections $\sigma:\cT\to\cCP$, via the Bers Double Uniformization Theorem.

% Let $(M,h)$ be a quasifuchsian hyperbolic manifold. Then it is isometric
% to $\bbH^3/\rho(\Gamma)$, where $\rho$ is a properly discontinuous action of
% $\Gamma$ on $H^3$ by isometries. Since hyperbolic isometries act as 
% complex projective transformations on $\C P^1$ identified with the boundary
% at infinity of $H^3$, the action of $\rho$ extends to a complex projective
% action, still denoted by $\rho$, of $\Gamma$ on $\partial_\infty H^3\setminus
% \Lambda_\Gamma$. So each connected component of $\partial_\infty M$ is 
% equipped with a complex projective structure, in particular a complex
% projective structure $\sigma_+(h)\in \cCP$ on $\partial_+M$. 

Via the natural forgetful map $p:\cCP\to \cT$, sending a complex projective
structure on $S$ to the underlying complex structure, given a quasifuchsian
hyperbolic metric $h\in\cQF$ on $M$, the boundary at infinity $\partial_\infty M$
is equipped with a complex structure, determined by a pair $(c_+,c_-)\in
\cTT$,
with $c_\pm$ corresponding to the complex structure on the boundary component 
$\partial_\pm M$. We will use the following well-known result.

\begin{theorem}[Bers \cite{bers}] \label{tm:bers}
For any $(c_+, c_-)\in \cTT$, there is a unique quasifuchsian metric 
on $M$ such that $c_-$ and $c_+$ are the complex structure at infinity on 
the boundary components of $M$.
\end{theorem}

Now consider a fixed complex structure $c_-\in\cTb$. Then, for each $c_+\in \cT$,
Theorem \ref{tm:bers} gives a unique quasifuchsian metric $h$ on $M$ with 
complex structure $c_\pm$ at $\partial_\pm M$. We call $\sigma_{c_-}(c_+)$
the corresponding complex projective structure on $\partial_+M$ defined by $h$.
This defines another section $\sigma_{c_-}:\cT\to\cCP$, for each choice of 
$c_-$.
The following proposition is proved and put in context in \cite{loustau:complex}.

\begin{prop} \label{pr:bers}
For all $c, c'\in\cTb$, $d(\sigma_{c'}-\sigma_{c})=0$, where
$\sigma_{c'}-\sigma_{c}$
is considered as a section of $T^{*(1,0)}\cT$, that is, a $1$-form over $\cT$.
\end{prop}

% \begin{defi}
Given $c_-\in \cT$, we define another symplectic structure on $\cCP$ as the 
pull-back
of 
the cotangent symplectic structure on $T^{*(1,0)}\cT$ by the map sending 
$\sigma\in
\cCP$
to $(p(\sigma),\sigma-\sigma_{c_-})\in T^{*(1,0)}\cT$. It follows from Proposition 
\ref{pr:bers} that this symplectic structure does not depend on the choice of
$c_-$.  
We will denote it by $\omega_B$.
% \end{defi}

On the other hand, the holonomy representation 
$hol:\cCP\to\cX(\pi_1S,\PSL(2,\C))$ 
is a local diffeomorphism between the moduli space of complex projective structures
$\cCP$ and the character variety $\cX(\pi_1S,\PSL(2,\C))$, which is equipped with
the Goldman symplectic structure $\omega_G$, obtained by taking the cup-product
of the cohomology classes with coefficients paired with the Killing form on $\psl(2,\C)$,
see \cite{goldman-symplectic}. Pulling back $\omega_G$ by $hol$ thus also gives a complex symplectic
structure, which we also call $\omega_G$, on $\cCP$. We will denote by 
$\omega_G^i$ the imaginary part of $\omega_G$, which is a (real) symplectic structure

% [Remark: the real part of the symplectic structure does not depend on the
% choice of a section. Equal to imaginary part of Goldman.]

The following theorems provide the relation between the Goldman symplectic 
structure and the pull-back of the cotangent bundle symplectic structure via 
the Fuchsian and Bers sections, see \cite[Corollary 5.13]{loustau:phd} and 
\cite{kawai} (see also \cite[Theorem 5.8]{loustau:phd} or 
\cite{loustau:complex} for the statement with the constant compatible with our 
notations, and a simpler proof.).

\begin{theorem}[Loustau]
$\omega_G=p^*\omega_{WP}+i\omega_F$, where $p:\cCP\to \cT$ is the canonical
forgetful map. 
In particular, $\omega_G^i=\omega_F^r$.
\end{theorem}

\begin{theorem}[Kawai \cite{kawai}]\label{tm:kawai}
$\omega_G=i\omega_B$.
\end{theorem}

Note that besides the Goldman symplectic structure, there are other complex symplectic structures on $\cCP$.
In fact it is known from Hitchin's work \cite{hitchin}
that there
is a hyperk\"ahler structure defined at least on an open subset of $\cCP$. We do not 
elaborate on this here, however understanding this hyperk\"ahler structure geometrically
can be one motivation for investigating the (complex) symplectic structures on $\cCP$
in relation to other moduli spaces of geometric structures.

\subsection{Globally hyperbolic AdS manifolds}

We have seen in Section \ref{ssc:ghads} that it is possible to 
identify $\cGH_{-1}$ with $\cTT$ and consider
on this moduli space the sum of the Weil-Petersson symplectic forms on the two
factors, $\omom$, which is a real symplectic form on $\cGH_{-1}$.

The left and right hyperbolic metrics $m_L,m_R$ defined by Theorem \ref{tm:mess} 
can also be obtained by considering
any ``well-behaved'' Cauchy surface, see \cite[Lemma 3.16]{minsurf}.

\begin{lemma}
Let $\Sigma$ be a Cauchy surface in $M$ with principal curvatures everywhere in 
$(-1,1)$. Then, up to isotopy,
$$m_L=I((E+JB)\cdot,(E+JB)\cdot), ~ m_R=((E-JB)\cdot,(E-JB)\cdot)~, $$
where $I$ and $B$ are the induced metric and shape operator of $\Sigma$,
respectively,
and $E$ is the identity map from $T\Sigma$ to itself.
\end{lemma}

%%\subsubsection{Convex cores of AdS manifolds}

A description of GHM AdS manifolds in terms of the geometry of 
the convex core, similar to that of quasifuchsian hyperbolic manifolds, is also 
available, see \cite{mess,mess-notes}. 
% The following proposition is the AdS 
% version of proposition \ref{pr:convex_core_hyp}.

\begin{prop}[Mess]
Let $(M,g)$ be a GHM AdS 3-manifold homeomorphic to $S\times \R$.
Then $M$ contains a unique smallest non-empty convex subset $C(M,g)$, which is 
furthermore compact. 
If $M$ is not Fuchsian, then $\partial C(M,g)$ is the disjoint union of two
surfaces 
homeomorphic to $S$, which we will denote by $S_+, S_-$. Each has a hyperbolic
induced
metric (denoted by $m_+, m_-$) and its pleating is described by a measured
lamination
(here $l_+, l_-$). 
\end{prop}

The relations between the induced metrics and bending laminations on the
boundary of the convex core and 
the left and right hyperbolic metrics are particularly simple.

\begin{theorem}[Mess \cite{mess}] \label{tm:mess-diagram}
With the notations above, $m_L=E_L(m_+,l_+)=E_R(m_-,l_-)$, while 
$m_R=E_R(m_+,l_+)=E_L(m_-,l_-)$.
\end{theorem}

Those simple relations are shown in Figure \ref{fg:mess}.
\begin{figure}[h] 
\begin{center}
\includegraphics[width=0.230\textwidth]{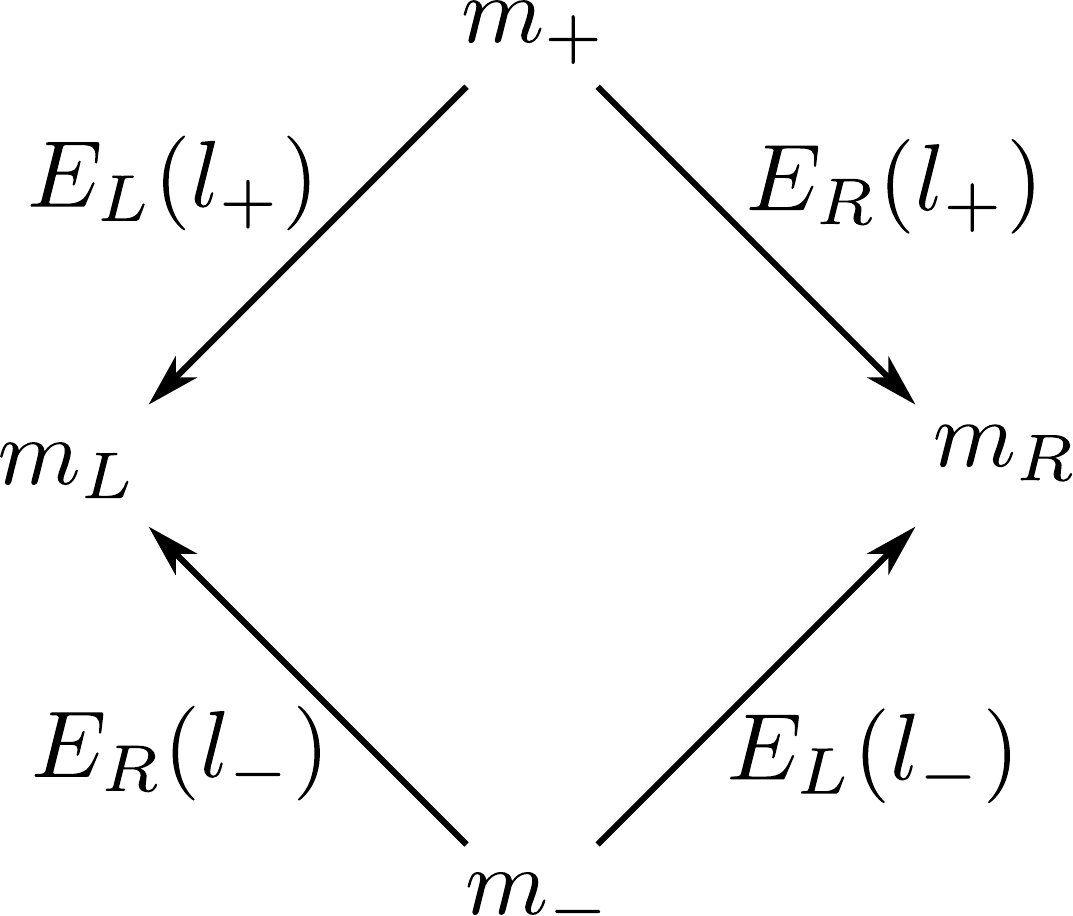}
\caption{Relation between the left/right metrics and the boundary of the convex
core.} \label{fg:mess}
\end{center}
\end{figure}
% \begin{figure} 
% \begin{center}
% \begin{tikzpicture}[node distance=3cm, auto]
%   \node (m+) {$m_+$};
%   \node (hL) [below left of=m+] {$m_L$};
%   \node (hR) [below right of=m+] {$m_R$};
%   \node (m-) [below right of=hL] {$m_-$};
%   \draw[->] (m+) to node[left] {$E_L(l_+)$} (hL);
%   \draw[->] (m+) to node[right] {$E_R(l_+)$} (hR);
%   \draw[->] (m-) to node[left] {$E_R(l_-)$} (hL);
%   \draw[->] (m-) to node[right] {$E_L(l_-)$} (hR);
% \end{tikzpicture}
% \end{center}
% \caption{Relation between the left/right metrics and the boundary of the convex
% core.} \label{fg:mess}
% \end{figure}

\subsection{Minimal and maximal surfaces}

There is a deep relationship between maximal surfaces in $AdS^3$ and 
harmonic maps and minimal Lagrangian maps. A key point is the following 
lemma due to Ayiama, Akutagawa and Wan \cite[Proposition 3.1]{AAW}.

\begin{lemma} \label{lm:max-hopf}
Let $g$ be a GHM AdS metric on $M$, and let $\Sigma$ be the (unique) closed
space-like maximal surface in $(M,g)$. Let $I$ and $\II$ be the induced
metric and second fundamental form on $\Sigma$, and let $m_L, m_R$ be the 
left and right hyperbolic metrics on $\Sigma$. The identity map from
$(\Sigma, [I])$ to $(\Sigma, m_L)$ (resp. to  $(\Sigma, m_R)$) is 
harmonic, and the imaginary part of its Hopf differential is equal to $\II$
(resp. to $-\II$).
\end{lemma}

\subsection{Globally hyperbolic flat and de Sitter manifolds}

% \subsubsection{Classification of flat manifolds}
The $3$-dimensional Minkowski space is defined as the space $\R^{2,1}$ 
with the flat Lorentzian metric of signature $(2,1)$.

GHM flat metrics on $M$ are defined in the same manner as in the AdS case described previously. 
Recall that we denote by $\cGH_0$ the moduli spaces of flat GHM metrics on $M$. 
We consider only future complete spacetimes, presenting an initial singularity. 
Past complete spacetimes are obtained by time reversal.

The isometry group $\isom_0(\R^{2,1})$ is isomorphic to a semi-direct product $\PSL(2,\R)\ltimes\psl(2,\R)$. Thus, the holonomy representations of GHM flat manifolds define points in the representation varieties $\Rep(\pi_1S,\PSL(2,\R)\ltimes\psl(2,\R))$. A holonomy representation then decomposes as $\rho=(\rho_0,\tau)$ with linear part $\rho_0:\pi_1(S)\to\PSL(2,\R)$ and a $\rho_0$-cocycle $\tau:\pi_1(S)\to\psl(2,\R)$. The following result of Mess \cite{mess,mess-notes} provides the classification of GHM flat metrics in terms of holonomies.

\begin{theorem}[Mess]\label{th:flat}
The linear part $\rho_0$ of the holonomy representations of a GHM flat metric 
have maximal Euler number, so that it is the holonomy representations of 
a hyperbolic structure $h_0\in \cT$.
Given $\rho_0\in \cT$ and a $\rho_0$-cocycle $\tau$, there is a unique 
future complete GHM Minkowski metric $h\in \cGH_0$
such that $\rho_0$ and $\tau$ describes its holonomy representation.
\end{theorem}

If $\tau$ is a coboundary, then the holonomy representation of $h$ is conjugate to $\rho_0$. 
Thus only the cohomology class of $\tau$ is relevant.

The first cohomology group $H^1(\pi_1S,\psl(2,\R)_{\Ad \rho_0})$ can be seen as the fibre of the cotangent bundle $T^*\cT$ over Teichm\"uller space. In fact, the embedding of $\cT$ into the $\PSL(2,\R)$ representation variety parametrizes the tangent space to $\cT$ at $\rho_0$ by the first cohomology group $H^1(\pi_1S,\psl(2,\R)_{\Ad \rho_0})$ and the non-degenerate cup product can be used as the duality pairing between $T\cT$ and $T^*\cT$. We thus have a one-to-one correspondence $hol:\cGH_0\to T^*\cT$ sending $h$ to $(\rho_0,\tau)$.

% \subsubsection{Classification of de Sitter manifolds}
The $3$-dimensional de Sitter space is defined as the set
$$dS^3=\{ x\in \R^{3,1}~|~\langle x,x\rangle=1\} $$
with the induced metric from the $4$-dimensional Minkowski metric.

We will denote by $\cGH_1$ the moduli spaces of de Sitter GHM metrics on $M$. Again, we consider only future complete spacetimes.

The isometry group $\isom_0(dS^3)$ is isomorphic to $\PSL(2,\C)$. 
The holonomy representations of GHM dS manifolds therefore define points 
in the character variety $\cX(\pi_1S,\PSL(2,\C))$. As for quasifuchsian manifolds, 
and more generally for hyperbolic ends, the classification of GHM de Sitter spacetimes in terms of holonomies is not possible since the map $hol:\cGH_1\to\cX(\pi_1S,\PSL(2,\C))$ is only a local diffeomorphism
(importantly it is not injective). However, similarly to hyperbolic ends, de Sitter manifolds can be understood in terms of a complex projective structure at their boundary at future infinity $\partial_+M$. More precisely, the developing map $dev:\tilde M\to dS^3$ restricts to a developing map $dev:\widetilde{\partial_+M}\to\partial_+ dS^3\simeq\C P^1$. The holonomy representation $hol:\pi_1S\to\PSL(2,\C)$ then endows $\partial_+M$ with a complex projective structure. We denote the map associating to a GHM dS manifold $(M,g)$ the corresponding complex projective structure on $\partial_+M$ by $\partial_+^{dS}:\cGH_1\to\cCP$. A result of Scannell \cite{scannell:de_sitter} gives the converse construction of GHM dS manifolds given a complex projective structure on $S$. We thus obtain the following result.

\begin{theorem}[Scannell]
GHM de Sitter spacetimes are in one-to-one correspondence with complex projective structures.
\end{theorem}

We continue to denote by $\omega_G^i$ the symplectic form on $\cGH_1$ obtained by pull-back of the imaginary part of the Goldman symplectic form on $\cCP$. 

\subsection{Initial singularities and projective duality}

There is another possible description of globally hyperbolic flat and de Sitter spacetimes
in terms of hyperbolic metrics and measured laminations on surfaces. We outline it here,
referring to \cite{mess,scannell:de_sitter,benedetti-bonsante} 
for proofs.

\subsubsection{Dual hyperbolic ends of de Sitter spacetimes}

The de Sitter geometry can also be understood in terms of hyperbolic ends via the duality between $dS^3$ and $\bbH^3$ coming from their simultaneous realization as embedded quadrics in $\R^{3,1}$. 
The dual relation between space-like $k$-planes and their orthogonal time-like $(4-k)$-planes through the origin of $\R^{3,1}$ induces a duality between $dS^3$ and $\bbH^3$, mapping points in one space to geodesic planes in the other and geodesic lines to geodesic lines. Further, the sphere at infinity $\bbS^2_\infty$, corresponding to null directions in $\R^{3,1}$, agrees with both $\partial_+ dS^3$ and $\partial_\infty\bbH^3$.

Thus, given a hyperbolic end $(E,h)$ consider its universal covering space $\tilde E$ and its image in $\bbH^3$ under the developing map. The set of geodesic planes contained in $\tilde E$ then defines, via duality, a set of points in $dS^3$ which determines a convex domain of dependence in $dS^3$. The holonomy representation of $h$ then acts properly discontinuously on such domain and the quotient space is a GHM de Sitter manifold $(M,g)$. It should be clear, in particular, that the corresponding complex projective structures at $\partial_\infty E$ and $\partial_+M$ agree.

\subsubsection{Initial singularities of Minkowski spacetimes}

To describe the analogous constructions in the case of flat spacetimes, 
we need to consider the geometry of their initial singularities, see \cite{mess,benedetti-bonsante}. 
The developing map of a GHM flat manifold $(M,g)$ is an embedding of $\tilde M$ into a convex future complete domain $dev(\tilde M)\subset\R^{2,1}$. Let $\Sigma$ be a Cauchy surface in $M$ and consider the restriction of the developing map to $\tilde\Sigma$. Then, $dev(\tilde M)$ can be described as the chronological future of the domain of dependence of $dev(\tilde\Sigma)$
$$ dev(\tilde M)=I^+(\cD(dev(\tilde\Sigma)))~. $$
In particular, $I^+(\overline{dev(\tilde M)})=dev(\tilde M)$. Also for any pair of subsets  $T,T'$ of $\overline{dev(\tilde M)}$ satisfying the condition $I^+(T)=I^+(T')=dev(\tilde M)$ their intersection $T\cap T'$ also satisfies $I^+(T\cap T')=I^+(T)\cap I^+(T')=dev(\tilde M)$ (in particular $T\cap T'$ is non-empty), so that there exists a unique smallest subset $T(\tilde M)\subset\overline{dev(\tilde M)}$ such that $dev(\tilde M)=I^+(T(\tilde M))$. This is the so-called initial singularity of $\tilde M$.

The initial singularity $T(\tilde M)$ has the structure of a $\R$-tree and is dual to a measured godesic lamination on $\bbH^2$ (identified with the set of future-pointing unit time-like vectors in $\R^{2,1}$). First, note that there is a well defined retraction $r:\tilde M\to T(\tilde M)$ sending each point $p\in\tilde M$ to the unique point $r(p)\in T(\tilde M)$ maximising the time separation from $dev(p)$. This then gives rise to a map $N:\tilde M\to\bbH^2$ sending each point $p\in\tilde M$ to the unit time-like vector
$$N(p)=\frac{dev(p)-r(p)}{|dev(p)-r(p)|}~.$$
To define the dual lamination to $T(\tilde M)$ associate to each point $p\in T(\tilde M)$ the set $F_p=N(r^{-1}(p))\subset\bbH^2$ and consider 
$$ \tilde l=\left(\bigcup_{{\substack{p\in T(\tilde M) \\ \dim F_p=2}}}\partial F_p\right) 
\cup \left(\bigcup_{{\substack{p\in T(\tilde M) \\ \dim F_p=1}}}F_p\right)~. $$
% \footnote{Carlos, I have re-introduced the dim 1 leaves since I think that otherwise
% you miss some leaves of the foliation! Is this correct?}\footnote{I think this might indeed be necessary for general laminations. In the case of multicurves I think it should be enough to consider only the boundaries of 2-dimensional leaves, but maybe something is missing in the limiting process for general laminations. In any case, introducing the 1-dim leaves is certainly more correct.}
The measure is then defined for transverse arcs $k$ in $\bbH^2$ by the distance in 
$T(\tilde M)$ between the points corresponding to the end points of $k$.

Conversely, given a measured lamination $l\in\cM\cL$ we can reconstruct the cocycles deforming Fuchsian representations into GHM flat representations. Given $\rho\in\cT$ consider $l\in\cML$ a measured geodesic lamination supported on a simple closed curve $\gamma$ with weight $a$. Consider the lift $\tilde l$ of $l$ to $\bbH^2$. Each leaf of $\tilde l$ is then a complete geodesic of $\bbH^2$ and we can consider, for each point $p\in\tilde l$, the infinitesimal generator $J_p$ of hyperbolic translations along the corresponding leaf. We then have a $\rho$-cocycle by
$$\tau(\gamma')=\sum_{p\in \gamma\cap \gamma'}a J_p~,\qquad \gamma'\in\pi_1S~. $$
The construction for general measured laminatios is then obtained by a limiting procedure. This thus defines a bijective map $\partial_*^{Mink}:\cGH_0\to\cT\times\cML$.
% \footnote{(JM) Note that 
% the sum here is usually infinite, however it's perhaps sufficient to leave this as it is
% rather than add more explanations?}\footnote{Perhaps it easier to describe this construction using multicurves, where it is well defined, and then argue that it extends nicely to general laminations.}

%%% Local Variables: 
%%% mode: latex
%%% TeX-master: "doublemaps"
%%% End: 

%% file: doublemaps3.tex
\section{Wick rotations and double maps} \label{sc:wick}

In this section we explain the relation between the three- and two-dimensional points of view 
developed in the
introduction. More specifically, we shall see why Theorem \ref{tm:wick-cc}
implies Theorem \ref{tm:double-earthquake}, and Theorem \ref{tm:minimal} is equivalent to
Theorem \ref{tm:harmonic}. We then prove that the double earthquake map $\cE$ and the double
harmonic map $\cH$ are one-to-one, leaving the discussion of the regularity properties of the earthquake map for the next section.

\subsection{Earthquakes and the boundary of the convex core}

Let us start considering the relations between Theorem \ref{tm:wick-cc} and Theorem
\ref{tm:double-earthquake}.
As we have seen in the introduction, the definition of the Wick rotation 
$W_\partial^{AdS}:\cHE\to\cGH_{-1}$ between hyperbolic ends and GHM AdS manifolds
is given by matching the boundary data at the initial boundary of a hyperbolic ends %(determined by
%its complex projective structure at infinity)
and at the upper boundary of the convex core of a GHM AdS manifolds
$$ W_\partial^{AdS}=(\partial_+^{AdS})^{-1}\circ\partial_+^{Hyp}~. $$
(Recall that the maps $\partial_+^{AdS}$ and $\partial_+^{Hyp}$ are defined in Section 
\ref{ssc:convexcores}).

The motivation behind this definition is quite clear in terms of 3-dimensional geometry. 
On the other hand, due to the lack of a smooth structure on $\cT\times\cML$, 
it is unclear how to use the Wick rotation $W_\partial^{AdS}$ to relate the geometric properties 
of the two moduli spaces. To address this we must describe the Wick rotation in terms of 
better behaved (smooth) maps.

First note that by Thurston's result, Theorem \ref{tm:grafting}, we have a relation 
between the complex projective data at the asymptotic boundary and the lamination data 
at the initial boundary of hyperbolic ends given by grafting
$$ \partial_\infty^{Hyp}=\cG\circ\partial_+^{Hyp}:\cHE\to\cCP~. $$
The smooth and symplectic structures on $\cCP$ can in fact be defined via pull-back 
the inverse of this map $\partial_\infty^{Hyp}$. 
Analogously, by Mess' result, Theorem \ref{tm:mess-diagram}, the holonomy mapping can be 
written in terms of the upper boundary of the convex core in GHM AdS manifolds 
via the double earthquake map
$$ hol=\cE\circ\partial_+^{AdS}:\cGH_{-1}\to\cTT~, $$
with the smooth and symplectic structures on $\cGH_{-1}$ also given via pull-back.

On the other hand, the composition $\cG'=\cG\circ \delta^{-1}$ of the grafting map $\cG$ with the inverse of
$\delta:\cT\times\cML\to T^*\cT$, the map sending $(m,l)$ to $d_mL(l)$, is a $C^1$ symplectomorphism between
$(T^*\cT,\omega_*)$ and $(\cCP,\omega_G^i)$, see \cite{cp}. This motivates us to consider the analogous
composition, $\cE'=\cE\circ \delta^{-1}$, of the double earthquake map $\cE$ with $\delta^{-1}$. We then
obtain be the diagram in Figure \ref{fg:wick-cc2}, which is shown below to be commutative.
\begin{figure}[h] 
\begin{center}
\includegraphics[width=0.55\textwidth]{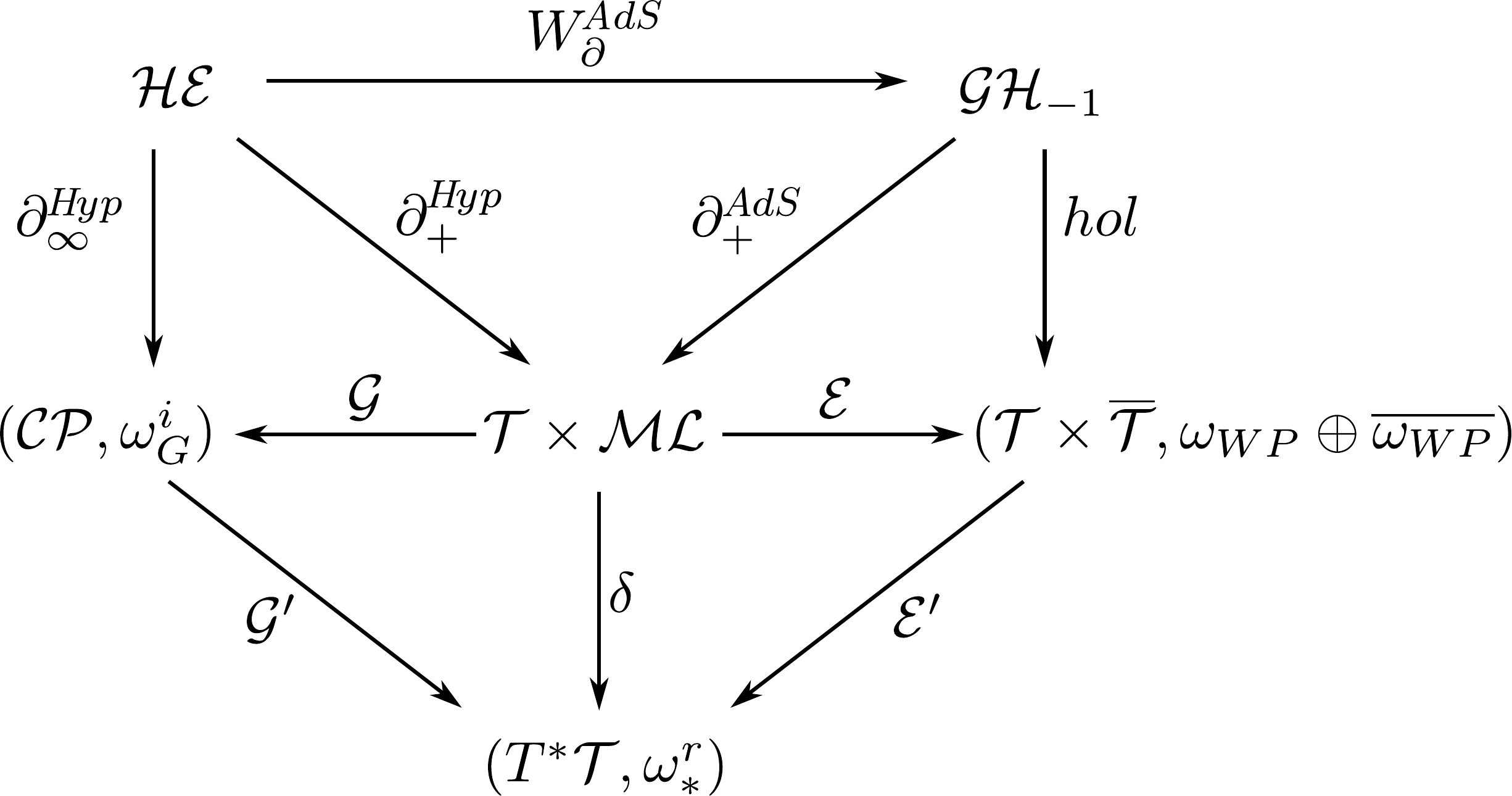}
\caption{Relation between double earthquakes and Wick rotations through pleated surfaces}
\label{fg:wick-cc2} 
\end{center}
\end{figure}

\begin{lemma} \label{lm:comm2}
The diagram in Figure \ref{fg:wick-cc2} commutes. 
\end{lemma}

\begin{proof}
The commutativity of the upper triangle follows directly from the definition of $W_{\partial}^{AdS}$,
while the definitions of $\cG'$ and $\cE'$ provides the commutativity of the two lower triangles.
The fact that the middle left triangle commutes is a translation of Thurston's Theorem
\ref{tm:grafting}, while the middle right triangle commutes by Mess' Theorem \ref{tm:mess-diagram}.  
\end{proof}

This allows us to write the relation between the Wick rotation and the double earthquake map as
$$ W_\partial^{AdS}=hol^{-1}\circ\cE'\circ(\cG')^{-1}\circ \partial_\infty^{Hyp}~. $$

We record the following consequence for future use.

\begin{remark} \label{rk:equiv-partial}
$W_\partial^{AdS}$ is $C^1$-smooth and symplectic if and only if $\cE'$ is $C^1$-smooth and symplectic.
\end{remark}

\subsection{Harmonic maps and minimal surfaces}

Turning now to the relations between Theorem \ref{tm:minimal} and Theorem \ref{tm:harmonic}, 
we shall use a much simpler commutative diagram, see Figure \ref{fig:wick-minimal}. 
From the introduction, the map $W_{min}:\cAF'\to\cGH_{-1}$ is defined 
by matching the holomorphic data of the minimal surface in an almost-Fuchsian 
manifold and the maximal surface of a GHM AdS manifold. More precisely, we have
$$ W_{min}=\max{}^{-1}\circ \min $$
where $\min:\cAF\to T^*\cT$ (resp. $\max:\cGH_{-1}\to T^*\cT$) is the map sending an 
almost-Fuchsian (resp. maximal globally hyperbolic AdS) metric on $M$ to 
the complex structure and holomorphic quadratic differential determined on its unique minimal 
(resp. maximal) surface by the first and second fundamental forms.

Considering also the maps $\partial_\infty^{Hyp}:\cAF'\to\cCP$ and $hol:\cGH_{-1}\to\cTT$
 we obtain the diagram Figure \ref{fig:wick-minimal}, which commutes as a direct consequence of Lemma 
\ref{lm:max-hopf}.
\begin{figure}[h] 
\begin{center}
\includegraphics[width=0.57\textwidth]{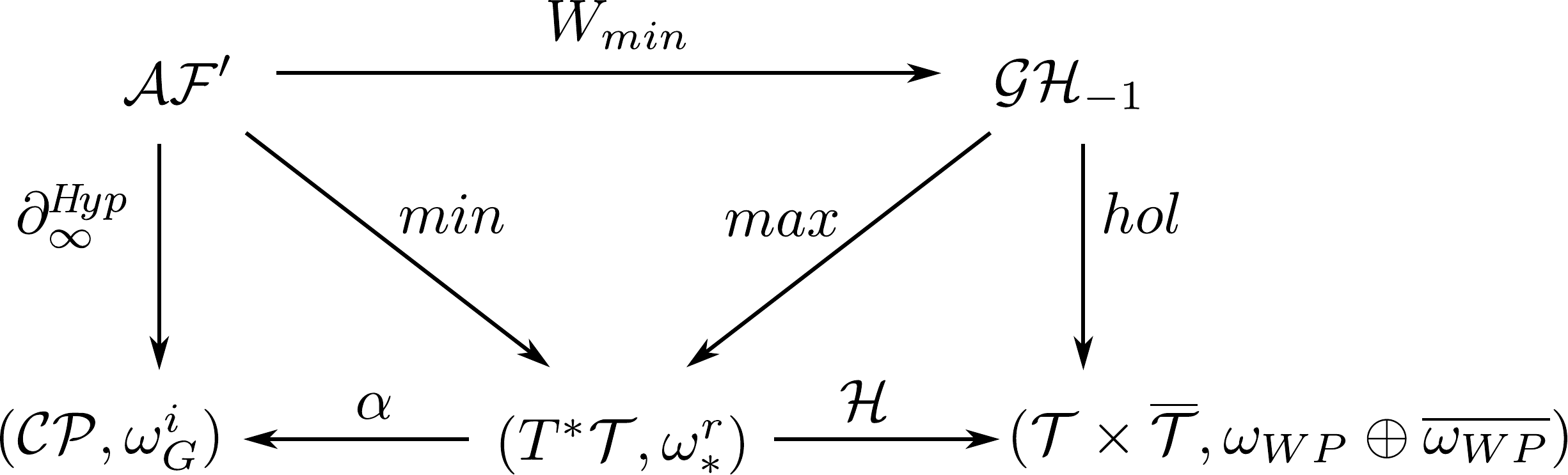}
  \caption{The minimal surfaces Wick rotation}
  \label{fig:wick-minimal}
\end{center}
\end{figure}
% \begin{center}
% \begin{tikzpicture}[node distance=3cm, auto]
%   \node (T*T) {$(T^*\cT,\omega_*)$};
%   \node (TT) [right of=T*T] {$(\cTT,\omom)$};
%   \node (AF) [above of=T*T] {$(\cAF,\omega_G^i)$};
%   \node (GH) [right of=AF] {$(\cGH,\omom)$};
%   \draw[->] (T*T) to node {$\cH$} (TT);
%   \draw[->] (GH) to node[right] {$(\rho_l,\rho_r)$} (TT);
%   \draw[->] (AF) to node {$W_{min}$} (GH);
%   \draw[->] (AF) to node[right] {$\min$} (T*T);
% \end{tikzpicture}
% \end{center}
% \begin{figure}[ht]
%   \centering
%   \begin{tikzpicture}[node distance=3cm, auto]
%   \node (AF) {$\cAF'$};
%   \node (CP) [left of=AF]{$(\cCP,\omega_G^i$)};
%   \node (GH) [right of=AF] {$\cGH_{-1}$};
%   \node (TT) [right of=GH]{$(\cTT,\omom)$};
%   \node (T*T1) [below of=AF] {$(T^*\cT,\omega_*)$};
%   \node (T*T2) [below of=GH] {$(T^*\cT,\omega_*)$};
%   \draw[->] (AF) to node {$W_{min}$} (GH);
%   \draw[->] (AF) to node[left] {$\min$} (T*T1);
%   \draw[->] (GH) to node {$\max$} (T*T2);
%   \draw[->] (T*T1) to node {} (T*T2);
%   \draw[->] (AF) to node[above] {$\partial_\infty^{Hyp}$} (CP);
%   \draw[->] (GH) to node {$hol$} (TT);
%   \draw[->] (CP) to node [below left]{$\alpha$} (T*T1);
%   \draw[->] (T*T2) to node[below right] {$\cH$} (TT);
% \end{tikzpicture}
%   \caption{The minimal surfaces Wick rotation}
%   \label{fig:wick-minimal}
% \end{figure}

The map $\alpha=\min\circ(\partial_\infty^{Hyp})^{-1}$ is symplectic 
up to a multiplicative constant, see \cite[Corollary 5.29]{loustau:phd}.

\begin{theorem}[Loustau] \label{tm:loustau-min}
$Re(\alpha^*\omega_*)=-\omega_G^i$.
\end{theorem}

We thus have the following remark.

\begin{remark} \label{rk:equiv2}
$\cH$ is symplectic (up to sign) if and only if $W_{min}$ is symplectic (up to sign). 
\end{remark}

Here ``up to sign'' means that, because of the minus sign in Theorem \ref{tm:loustau-min},
one map is symplectic if and only if the minus the other is symplectic.

\begin{proof}
If $\cH$ is symplectic, then it follows from Figure \ref{fig:wick-minimal} that
$-W_{min}$ is symplectic, because it can be written as a composition of symplectic
maps.

For the converse note that both $\cH$ and $W_{min}$ are real analytic.% \footnote{add argument for this!}.
If $W_{min}$ is symplectic, it follows from the diagram that 
$\cH$ is symplectic on an open subset of $T^*\cT$. Since the symplectic forms on 
both $T^*\cT$ and $\cTT$ are analytic, it follows that 
$\cH$ is symplectic everywhere. 
\end{proof}

\subsection{The double maps are one-to-one and onto}

This part contains the (simple) proofs that the double earthquake map and 
the double harmonic map are one-to-one. 

\begin{lemma}
The map $\cH:T^*\cT\to \cTT$ is bijective.
\end{lemma}

\begin{proof}
Let $(m_L,m_R)\in \cTT$. There is then a unique minimal Lagrangian diffeomorphism
isotopic to the identity $\phi$ from $(S, m_L)$ to $(S, m_R)$, see \cite[Corollaire 2.3.4]{L5}
or \cite{schoen:role}. If we define $m=m_L+\phi^*(m_R)$ and denote by $[m]$ its underlying
conformal structure, then $Id:(S,c)\to (S,m_L)$ and $\phi:(S,c)\to (S,m_R)$ are harmonic with
opposite Hopf differentials $-iq$ and $iq$. Therefore, $(m_L, m_R)=\cH(c,q)$, where $c$ is the
complex structure on $S$ associated to $[m]$. So $\cH$ is onto.

Conversely, consider $(c,q)\in T^*\cT$, that is, $c$ is a complex structure on $S$ and
$q$ is a holomorphic quadratic differential on $(S,c)$. Then there exists 
by Theorem \ref{tm:existence-harmonic}
a unique hyperbolic metric $m_L$ on $S$ such that the unique harmonic map $\phi_L:(S,c)\to (S,m_L)$ 
isotopic to identity has Hopf differential $-iq$, and there exists a
unique hyperbolic metric $m_R$ on $S$ such that the unique harmonic map $\phi_R:(S,c)\to (S,m_R)$ isotopic
to identity has Hopf differential $iq$. Then $\phi_R\circ \phi_L^{-1}:(S,m_L)\to (S,m_R)$ 
is minimal Lagrangian. This shows that $(c,q)$ is obtained from $(m_L,m_R)$ by the construction
in the first part of the proof, and this shows that $\cH$ is injective.
\end{proof}

\begin{lemma}
The double earthquake map $\cE:\cTML\to \cTT$ is bijective.
\end{lemma}

\begin{proof}
Let $(m_L, m_R)\in \cTT$. By Thurston's Earthquake Theorem (see the appendix in
\cite{kerckhoff}) there exists a unique $l\in \cML$ such that $m_L=E_L(m_R,2l)$.
But $E_R(l)=E_L(l)^{-1}$ and $E_L(2l)=E_L(l)^2$. So, if we set $m=E_L(m_r,l)$,
we have
$$ m_L=E_L(m,l),\quad m_R=E_R(m,l) $$
so that $(m_L, m_R)=\cE(m,l)$.

Conversely, if $(m_L,m_R)=\cE(m',l')$, then $m_L = E_L(m_R,2l')$, so it follows
from the uniqueness in the Earthquake Theorem that $2l'=2l$, so that $l=l'$ and 
$m=m'$.
\end{proof}

\subsection{Wick rotations to flat and dS manifolds}

We now consider analogous Wick rotations from hyperbolic ends to GHM flat and de Sitter manifolds.

\subsubsection{Hyperbolic metrics and measured laminations}

In analogy to the AdS case, we consider Wick rotations from hyperbolic ends to 
GHM flat manifolds $W^{Mink}_{\partial}:\cH\cE\to\cGH_0$ given by matching 
the data at the inital boundary of hyperbolic ends to the pair formed by 
the linear holonomy and the measured lamination dual to the initial singularity of 
GHM flat manifolds
$$ W^{Mink}_{\partial}:(\partial_*^{Mink})^{-1}\circ\partial_+^{Hyp}~. $$
Again, using the fact that the cocycle part of the holonomy is related 
to the measured lamination via grafting, we may write
$$ hol=\cG_0\circ\partial_*^{Mink}:\cGH_{0}\to T^*\cT~. $$
The smooth and symplectic structures on $\cGH_{0}$ are again given via pull-back.
We now obtain the first diagram in Figure \ref{fg:wick-flat-ds}, 
where we denote $\cG'_0=\cG_0\circ \delta^{-1}$.

The passage $W^{dS}_{\partial}:\cH\cE\to\cGH_1$ from hyperbolic ends 
to GHM dS manifolds is given automatically via duality, 
by matching the data at their common asymptotic boundary
$$ W^{dS}_{\partial}=(\partial_\infty^{dS})^{-1}\circ\partial_\infty^{Hyp}~. $$
Here there is no problem with differentiability and the symplectic structures agree, since in both cases the smooth and symplectic structures are again given via pull-back from $\cCP$. The second diagram in Figure \ref{fg:wick-flat-ds} describe these relations.

\begin{figure}[h] 
\begin{center}
\includegraphics[width=0.93\textwidth]{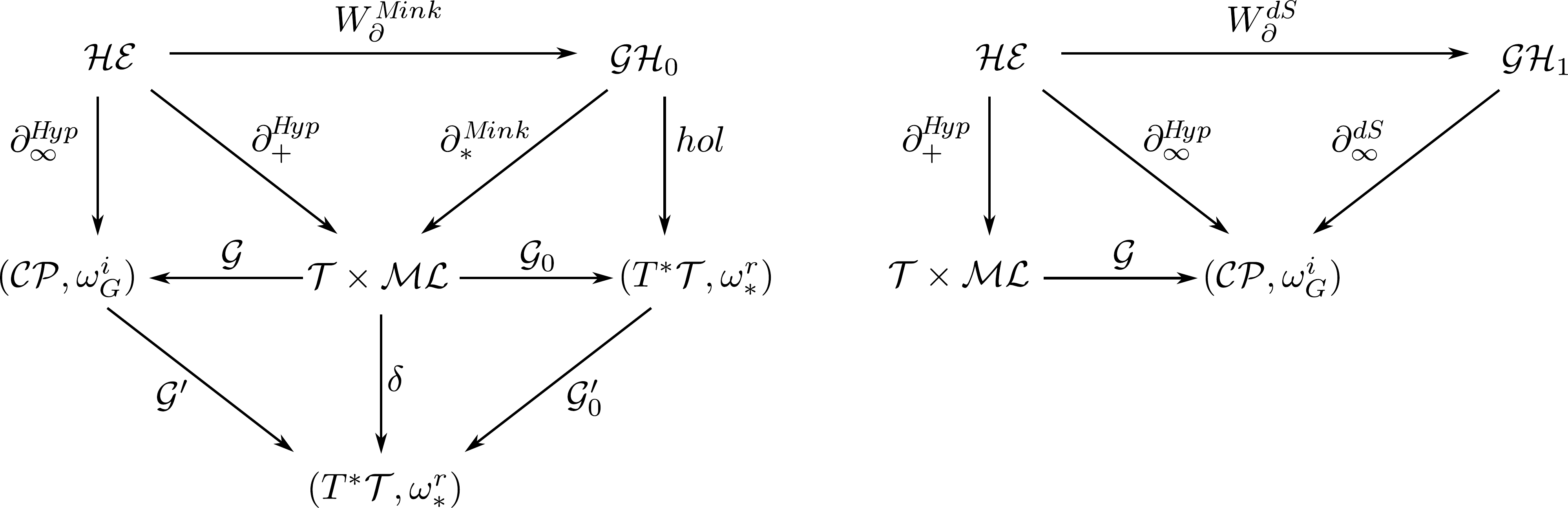}
\caption{Wick rotations to flat and de Sitter manifolds}
\label{fg:wick-flat-ds}
\end{center}
\end{figure}

Note that the diagrams in Figure \ref{fg:wick-flat-ds} commute, by definition of
the some of the maps used, as well as by Theorem \ref{tm:grafting} (for
the middle left triangle of the left diagram and the lower triangle of the
right diagram).

\subsubsection{CMC surfaces}

GHMC flat and de Sitter manifolds are also shown to admit a unique foliation by CMC surfaces.

\begin{theorem}[Barbot, B\'eguin, Zeghib \cite{BBZ}] \label{tm:flat_ds:foliation}
Any GHM flat and dS manifolds admit a unique foliation by closed space-like CMC surfaces,
with mean curvature in
\begin{itemize}
 \item $(-\infty,0)$, in the flat case,
 \item $(-\infty,-2)$, in the dS case.
\end{itemize}
For every prescribed $H$ as above, the spacetimes contain a unique closed space-like CMC-$H$ surface.
\end{theorem}

As in the AdS case, the first and second fundamental forms of the CMC-$H$ surface are in correspondence with a point in $T^*\cT$ (see \cite{Moncrief} and \cite[Lemma 6.1]{minsurf}).

\begin{prop}
Let $H\in (-\infty,-2)$.
Given a complex structure $c$ and a holomorphic quadratic
differential $q$ for $c$ on $S$, there is a unique GHM dS metric $h$ on $M$ such
that the induced metric and traceless part of the second fundamental form on the unique CMC-$H$ surface in $(M,h)$ is $I,\II_0$ with $I$ compatible with $c$ and $\II_0=Re(q)$.
\end{prop}

We may therefore construct as a version of the flat and de Sitter CMC-Wick rotation.
\begin{defi}
Let $H\in (-2,2)$, $H'\in(-\infty, 0)$ and $H''\in(-\infty,-2)$.
For each $h\in \cAF'$, let $S_H$ be the
unique closed CMC-$H$ surface in $(M,h)$, let $c$ be the conformal class of its 
induced metric, and let $q$ be the traceless part of its second fundamental form. 
There is then a unique GHM flat metric $h'$ and a unique GHM dS metric $h''$ on $M$ such that the (unique) CMC-$H'$ surface
in $(M,h')$ and the unique CMC-$H''$ surface in $(M,h'')$ have induced metric conformal to $c$ and the traceless part of its second fundamental
form is equal to $q$. We denote these maps respectively by $W_{H,H'}^{Mink}:\cAF'\to \cGH_{0}$ and $W_{H,H''}^{dS}:\cAF'\to \cGH_{1}$.
\end{defi}

%%% Local Variables: 
%%% mode: latex
%%% TeX-master: "doublemaps"
%%% End: 

%% file: doublemaps4.tex
\section{Regularity of the earthquake map}\label{sc:reg_earthquake}

We now focus on the $C^1$ regularity of the earthquake map, more specifically
on the proof of Proposition \ref{pr:EC1} and of Corollary \ref{cr:EC1}.
The notations here are similar to those of \cite[Section 2.5]{cp}, with the relevant adaptations, further developing some of the arguments which in \cite{cp} were too elliptic. As in \cite{cp}, the arguments will be based on the ideas and tools developed by Bonahon \cite{bonahon-toulouse,bonahon-ens}.% (recalled briefly in \cite[Section 2.5]{cp}).

\subsection{Maximal laminations and transverse cocycles}
\label{ssc:maxlam}

We first recall basic facts on transverse cocycles on a surface, which will be used
to give a parametrization of both the Teichm\"uller space $\cT$ and the space of 
measured geodesic laminations $\cML$, see \cite{bonahon-toulouse}.

We start with a fixed reference hyperbolic structure $m\in\cT$ on $S$ and a maximal
geodesic lamination $\lambda\in\cL$ on $(S,m)$. The maximality condition here is given with
respect to inclusion. Equivalently, this condition can be stated as the property that the complement
of $\lambda$ on $S$ is given by finitely many disjoint ideal triangles, see \cite{bonahon-toulouse}.

\begin{defi}
A transverse cocycle $\sigma$ for a lamination $\lambda$ is a function on arcs transverse to
$\lambda$ which is
\begin{itemize}
\item additive: $\sigma(k_1\sqcup k_2)=\sigma(k_1)+\sigma(k_2)$, 
\item $\lambda$-invariant: $\sigma(k_1)=\sigma(k_2)$ if $k_1$ and $k_2$ are
homotopic through a family of arcs transverse to $\lambda$.
\end{itemize}
We denote $\cH(\lambda,\R)$ the space of all transverse cocycles for $\lambda$.
\end{defi}

The space $\cH(\lambda,\R)$ has the structure of a finite dimensional vector space. 
In particular, if $\lambda$ is a maximal lamination, its dimension is given by
$\dim\cH(\lambda,\R)=6g-6$.

Note that the notion of transverse cocycles on maximal laminations generalizes the notion of
measured laminations. In fact, the support of any measured lamination $l\in\cML$ is
contained (possibly non-uniquely) into a maximal lamination $\lambda$ on $S$. Further,
given such maximal lamination $\lambda$ containing the support of $l$, the transverse
measure of $l$ defines uniquely a non-negative transverse cocycle $\mu$ on $\lambda$.
Thus any measured lamination gives rise to a non-negative transverse cocycle on some 
maximal lamination on $S$. Conversely, a non-negative transverse cocycle can be equally seen
as a transverse measure on the maximal lamination, thus defining a measured lamination.
This gives a 1-to-1 correspondence between $\cML\big|_\lambda$, the space of measured
laminations supported on $\lambda$, and $\cH(\lambda,\R_+)$, the space of non-negative 
transverse cocycles on $\lambda$.

It is also possible to give a parametrization the Teichm\"uller space in terms of transverse cocycles.
Given a maximal lamination $\lambda$ on $S$, Bonahon \cite{bonahon-toulouse} defines for each hyperbolic metric
$m\in\cT$ a transverse cocycle $\sigma_m\in\cH(\lambda,\R)$, assigning to each transverse arc
$k$ to $\lambda$ a real number $\sigma_m(k)$ which we now define. Let $\tilde\lambda$ be
the preimage of $\lambda$ in the universal cover $\tilde S$ of $S$. The maximality condition
for $\lambda$ then implies that $\tilde\lambda$ determines a tessellation of $\tilde S$ by
ideal triangles. For any pair $P,Q$ of such ideal triangles we associate a real number $\sigma_{PQ}$
as follows. Assuming, first, that $P$ and $Q$ are adjacent, we take $\sigma_{PQ}$ to
be the logarithm of the cross-ratio of the ideal quadrilateral defined by $P$ and $Q$. Equivalently,
$\sigma_{PQ}$ is the signed hyperbolic distance along their common edge between the orthogonal projections
of the opposite vertices to this edge.
For non-adjacent ideal triangles $P,Q$ we then define $\sigma_{PQ}$ as the sum of $\sigma_{P'Q'}$ over all pairs of adjacent ideal triangles $P',Q'$ between $P$ and $Q$. Note that such sum may be an infinte sum. However, an upper bound for each of the $\sigma_{P'Q'}$, given by the distance between their outermost edges \cite{bonahon-toulouse}, implies that $\sigma_{PQ}$ differs from the distance between the innermost edges of $P$ and $Q$ only by a finite constant, so that $\sigma_{PQ}$ is indeed well defined.
% \footnote{I don't think this should be a finite sum in general, so a limiting 
% argument should be added??}
% \footnote{Yes, you are right! In Bonahon's paper ``Shearing hyperbolic surfaces'' (Lemma 8)
% there is an upper bound for each of the $\sigma_{P_{i-1}P_i}$ given by the distance between
% their outermost edges. Therefore the sum, which may indeed be infinte,
% is bounded by the distance between the innermost edges of $P$ and $Q$ (plus some fininte
% constant).}

The transverse cocycle $\sigma_m\in\cH(\lambda,\R)$ associated to the hyperbolic metric $m\in\cT$ can now be defined. Given a transverse arc $k$ to $\lambda$ let $\tilde k$ be a lift of $k$ to $\tilde S$. By
transversality the endpoints of $\tilde k$ belong to the interior of ideal
triangles $P$ and $Q$ and we can define $\sigma_m(k)=\sigma_{PQ}$.
\begin{theorem}[Bonahon \cite{bonahon-toulouse}]
The map $\varphi_\lambda:\cT(S)\to\cH(\lambda,\R)$ defined by
$$\varphi_\lambda(m)=\sigma_m$$
is injective and open. Furthermore, it is real analytic into its image.
\end{theorem}

\subsection{Smoothness of the double earthquake}
\subsubsection{Differentiability}
We now turn to the $C^1$-smoothness of the double earthquake map $\cE'=\cE\circ\delta^{-1}:T^*\cT\to\cTT$, starting with the differentiability of $E^L\circ \delta^{-1}$. The strategy here is the same as in \cite{cp} showing that for each maximal lamination $\lambda$ there is a pair of tangentiable maps $\Phi_\lambda:\cT\times \cH(\lambda,\R_+)\to\cT$ and $\Psi_\lambda:\cT\times\cH(\lambda,\R_+)\to T^*\cT$ such that
\begin{itemize}
 \item the composition $\Phi_\lambda\circ\Psi_\lambda^{-1}$ agrees with $E^L\circ\delta^{-1}$ on  $\delta(\cT\times\cML|_\lambda)\subset T^*\cT$;
 \item for two maximal laminations, $\lambda$ and $\lambda'$, the tangent maps of $\Phi_\lambda\circ\Psi_\lambda^{-1}$ and $\Phi_{\lambda'}\circ\Psi_{\lambda'}^{-1}$ agree on $T_{(m,u)}T^*\cT$ for all $(m,u)\in\delta(\cT\times\cML|_\lambda\cap \cML|_{\lambda'})$.
\end{itemize}

Start by noting that given a maximal lamination $\lambda$ the notion of length of measured laminations and of earthquakes along measured laminations naturally extend to notions of length of transverse cocycles and shearings along transverse cocycles \cite{bonahon-toulouse}. Further, such extentions are well behaved under the vector space structure of $\cH(\lambda,\R)$ in that the length function $L:\cT\times\cH(\lambda,\R_+)\to\R$ is linear in its second argument and the shear map $E:\cT\times\cH(\lambda,\R_+)\to\cT$ satisfies the following equivariance property
\begin{align}\label{shear_prop}
E_{\sigma+\sigma'}(m)=E_\sigma\circ E_{\sigma'}(m).
\end{align}
% $$E_{\sigma_{m''}-\sigma_{m'}}\circ E_{\sigma_{m'}-\sigma_m}(\sigma_m)=E_{\sigma_{m''}-\sigma_{m'}}(\sigma_{m'})$$
It is thus natural to consider the following tangentiable maps
$$\Phi_\lambda(m,\sigma)=E_{\sigma}(m),\qquad \Psi_\lambda(m,\sigma)=d_mL(\sigma).$$

Given $m\in\cT$ and $u\in T^*_m\cT$ let $(m,l)=\delta^{-1}(m,u)\in\cT\times\cML$ denote the image of $(m,u)$ under the inverse of $\delta$. Then, choose a maximal lamination $\lambda$ containing the support of $l$ and let $\sigma\in\cH(\lambda,\R_+)$ denote the positive transverse cocycle corresponding to the measure of $l$. It follows directly from the definitions of length and shears that
$$\Phi_\lambda\circ\Psi_\lambda^{-1}(m,u)=\Phi_\lambda(m,\sigma)=E^L(m,l)=E^L\circ\delta^{-1}(m,u).$$
Further, from the equivariance of $E_\sigma(m)$ and the linearity of $L_m(\sigma)$, we can easily compute
$$d_{(m,\sigma)}\Phi_\lambda(\dot m,\dot\sigma)%=\frac{d}{dt}\Big|_{t=0^+}\Phi_\lambda(m_t,\sigma_t)
=\frac{d}{dt}\Big|_{t=0^+}E_{td_m\varphi_\lambda(\dot m)}\circ E_{t\dot\sigma}\circ E_{\sigma}(m)=(e_{d_m\varphi_\lambda(\dot m)}+e_{\dot\sigma})(E_\sigma(m))=d_{m}E_\sigma(e_{d_m\varphi_\lambda(\dot m)}(m)+e_{\dot\sigma}(m)),$$
where $e_\sigma(m)\in T_m\cT$ is the infinitesimal shearing vector at $m$ determined by $\sigma$, and
$$d_{(m,\sigma)}\Psi_\lambda(0,\dot\sigma)=\frac{d}{dt}\Big|_{t=0^+}d_mL(t\dot\sigma+\sigma)=d_mL(\dot\sigma)=e_{\dot\sigma}^*(m),$$
where $*$ means the duality between $T^*_m\cT$ and $T_m\cT$ with respect to the Weil-Petersson symplectic form. Note that here $d\Phi_\lambda$ and $d\Psi_\lambda$ denote the tangent maps of $\Phi_\lambda$ and $\Psi_\lambda$ and not their differentials.

To compute the differential of $\Phi_\lambda\circ\Psi_\lambda^{-1}$ we introduce a decomposition of the tangent space to $T^*\cT$ at $(m,u)$ into horizontal and vertical subspaces
$$T_{(m,u)}T^*\cT=H_{(m,u)}T^*\cT\oplus V_{(m,u)}T^*\cT.$$
First note that the map $\delta$ evaluated at a fixed measured lamination $l$ determines a section $s_l=\delta(\,\cdot\,,l):\cT\to T^*\cT$ of the cotagent bundle over $\cT$. This is in fact a smooth section since the Hessian of the length function of $l$ depends continuously on both $m$ and $l$, as follows for instance from \cite[Theorem 1.1]{wolf:hessian}.
We can then define the horizontal and vertical subspaces as
$$H_{(m,u)}T^*\cT=\{U^h=d_m s_l(\dot m);\,\dot m\in T_m\cT\},\qquad V_{(m,u)}T^*\cT=\{U^v=\dot u;\,\dot u\in T_m^*\cT\}.$$

A simple computation now gives for a horizontal vector $U^h\in H_{(m,u)}T^*\cT$
\begin{align*}
d_{(m,u)}(\Phi_\lambda\circ\Psi_\lambda^{-1})(U^h)=\frac{d}{dt}\Big[\Phi_\lambda\circ\Psi_\lambda^{-1}\circ s_l\circ\pi(m(t),u)\Big]=\frac{d}{dt}\Big[\Phi_\lambda\circ\Psi_\lambda^{-1}\circ s_l(m(t))\Big]
\cr
=d_m(\Phi_\lambda\circ\Psi_\lambda^{-1}\circ s_l)(\dot m)=d_mE_\sigma(\dot m)=d_mE_l^L(\dot m),
\end{align*}
with $\dot m=d_{(m,u)}\pi(U^h)$, and for a vertical vector $U^v\in V_{(m,u)}T^*\cT$
\begin{align*}
d_{(m,u)}(\Phi_\lambda\circ\Psi_\lambda^{-1})(U^v)=\frac{d}{dt}\Big[\Phi_\lambda\circ\Psi_\lambda^{-1}(m,u(t))\Big]=\frac{d}{dt}\Big[E_{\Psi_\lambda^{-1}(m,u(t))}(m)\Big]
\cr
=d_mE_\sigma(e_{\dot\sigma}(m))
=d_mE_\sigma(\dot u^*)=d_mE_l^L(\dot u^*),
\end{align*}
with $\dot u=U^v$ and $\dot\sigma=d_{(m,u)}(pr_2\circ\Psi_\lambda^{-1})(\dot u)$.
% $$d_{(m,u)}(\mathrm{pr_2}\circ\psi_\lambda^{-1})U^v=\dot\sigma$$
% $$e_{\dot\sigma}^*(m)=d_\sigma(\psi_\lambda(m,\,\cdot\,))(\dot\sigma)=d_{(m,u)}(\psi_\lambda(m,\,\cdot\,)\circ\mathrm{pr_2}\circ\psi_\lambda^{-1})U^v=U^v$$
This shows in particular that $d(\Phi_\lambda\circ\Psi_\lambda^{-1})$ does not depend on $\lambda$, since the right-hand sides of both equations are completely independent on its choice, implying that $E^L\circ \delta^{-1}$ is differentiable at each
point $(m,u)\in T^*\cT$ with
\begin{equation}\label{eq:differential}
d_{(m,u)}(E^L\circ \delta^{-1})(U)=d_mE_l^L(\dot m + \dot u^*)~.
\end{equation}

\subsubsection{Continuity of the differential}
To complete the argument, it now only remains to show that the differential of $E^L\circ \delta^{-1}$ is continous. Let $\alpha_{(m,l)}:T_{(m,u)}T^*\cT\to T_m^*\cT$ denote the projection onto the vertical subspace of $T_{(m,u)}T^*\cT$, sending $U$ to $\dot u$. To prove that $E\circ \delta^{-1}$ is $C^1$, it is sufficient to prove that $\alpha_{(m,l)}$ vary
continuously with $(m,l)$, since all other maps entering the right-hand side
of \eqref{eq:differential} are clearly smooth by \cite{kerckhoff:analytic} and the
analyticity of the Weyl-Petersson symplectic form.

On the other hand, the decomposition of $T_{(m,u)}T^*\cT$ into horizontal and vertical subspaces then allows us to explicitly write $\alpha_{(m,l)}$ as
$$\alpha_{(m,l)}=\id-d_m s_l\circ d_{(m,u)}\pi,$$
where $\id$ is the identity map in $T_{(m,u)}T^*\cT$, $d_ms_l$ denote the linear horizontal embedding of $T_m\cT$ into $T_{(m,u)}T^*\cT$ and $d_{(m,u)}\pi$ the natural projection of $T_{(m,u)}T^*\cT$ onto $T_m\cT$. So $\alpha_{(m,l)}$ depends 
continuously on $(m,l)$ and this concludes the proof that 
$E\circ \delta^{-1}$ is $C^1$.

\subsubsection{Proof of Corollary \ref{cr:EC1}}

It follows from Proposition \ref{pr:EC1} that $\cE\circ\delta^{-1}$ is $C^1$.

The map $\cE:\cT\times \cML\to \cT\times \cT$ is clearly a bijection, because
a GHM AdS manifold is uniquely determined by the induced metric and measured
pleating lamination on the upper boundary of the convex core, and any hyperbolic
metric and pleating lamination can be realized in this way. 
The map $\delta:\cT\times \cML\to T^*\cT$ is also bijective, see \cite{cp}.
So $\cE\circ \delta^{-1}$ is bijective.

Finally note that the differential of $\cE\circ\delta^{-1}$ is everywhere
non-singular, since GHM AdS manifolds are locally uniquely determined 
by the induced metric and measured pleating lamination on the upper boundary
of the convex core.

%%% Local Variables: 
%%% mode: latex
%%% TeX-master: "doublemaps"
%%% End: 

%% file: doublemaps5.tex
\section{Double maps are symplectic} \label{sc:double-earthquake}

In this section we provide proofs for the symplecticity of the double earthquake and double harmonic maps,
Theorem \ref{tm:harmonic} and Theorem \ref{tm:double-earthquake}.

\subsection{Train Tracks and the Thurston intersection form}

We start by recalling here another set of tools that will be needed in the next part of this
section. More details can be found e.g. in \cite{penner-harer} and \cite{sozen-bonahon}.

First let's introduce the notion of a train track carrying a lamination. A train
track 
$T$ on the surface $S$ is a (regular) tubular neighborhood of an
embedded smooth graph with at least 2-valent vertices. We shall consider only
generic train tracks with only 3-valent vertices. The edges of $T$ meet
tangentially at
vertices and, therefore, we may divide edges incident to a given
vertex as incoming or outgoing according to the relative direction of their
tangent vectors. We denote by $e_v$ the incoming edge and by $e_v^+,e_v^-$ the
outgoing edges
of a vertex $v$, where the $+$ and $-$ signs denote the order of the outgoing
edges with respect to the incoming one given by a fixed choice of orientation of
the surface.

An edge weight system for $T$ is a map $a:E(T)\to\R$ assigning a weight
$a(e)\in\R$ to each edge $e\in E(T)$ and satisfying the switch relation
$$ a(e_v)=a(e_v^+)+a(e_v^-) $$ 
for each vertex $v\in V(T)$. We denote by $\cW(T)$ the vector space of
edge weight systems for $T$.

A lamination $\lambda$ is said to be carried by a train track $T$ if it is
contained in its interior in such a way that the leaves of $\lambda$ are
transverse to the normal fibers of $T$. In the particular case of a maximal
lamination $\lambda$,
there is a 1-to-1 correspondence between transverse cocycles
$\sigma\in\cH(\lambda,\R)$ and edge weight systems $a\in\cW(T)$ obtained by
assigning to each edge $e\in E(T)$ the weight
$$a(e)=\sigma(k_e)$$
where $k_e$ is any normal fibre of $T$, see \cite{sozen-bonahon}.%\footnote{We should add a reference here.} 
The swich relation is automatically satisfied due to the additivity of $\sigma$.
We thus obtain a map $\cH(\lambda,\R)\to\cW(T)$ which is shown to be an isomorphism of
vector spaces. % therefore, we also have a map $\psi_T:\cT(S)\to\cW(T)$
% $$m\mapsto \psi(m)=a_m$$
% which is again injective and open.

% Given a hyperbolic metric $m$ and a maximal lamination $\lambda$, let us
% introduce the $m$-length $L_m:\cH(\lambda,\R)\to\R$ of transverse cocycles on
% $\lambda$. We first fix a train track $T$ carrying $\lambda$ and define
% $$L_m(\sigma)=\sum_{e\in E(T)}\int_{\lambda\cap
% k_e}L_m(\lambda_e(x))\sigma(x)dx,$$
% where $k_e$ is a normal fibre of $T$ through the edge $e$, $\lambda_e(x)$ is
% the leaf of $\lambda$ in the tubular neighborhood of the edge $e$ passing
% through the point $x\in\lambda\cap k_e$, $L_m(\lambda_e(x))$ is the $m$-length
% of the curve $\lambda_e(x)$ and $\sigma(x)dx$ is the ($\lambda$-transverse) line
% element of $\sigma$. Note that if $\sigma$ is a transverse measure, then its
% $m$-length is always positive.

The Thurston intersection form on $\cH(\lambda,\R)$ defined by
$$\Omega_{\text{Th}}=\sum_{v\in V(T)}da(e_v^+)\wedge da(e_v^-).$$
More precisely, given $\sigma,\sigma'\in\cH(\lambda,\R)$, let $a,a'\in\cW(T)$ be the
corresponding edge weight systems. Then 
$$\Omega_{\text{Th}}(\sigma,\sigma')=\sum_{v\in
V(T)}\Big(a(e_v^+)a'(e_v^-)-a'(e_v^+)a(e_v^-)\Big).$$

This gives a non-degenerate 2-form on $\cH(\lambda,\R)$ which is closely related with the $m$-length of transverse cocycles, see \cite{bonahon-toulouse}.
% \footnote{Can we add a reference for this non-trivial statement?}
Namely,
given a hyperbolic metric $m$ and $\sigma$ a transverse cocycle, the $m$-length of
$\sigma$ can be computed as value of the Thurston intersection between $\sigma_m$
and $\sigma$
$$L_m(\sigma)=\Omega_{\text{Th}}(\sigma_m,\sigma).$$
% This shows in particular that a necessary condition for a transverse cocycle $\sigma$ to be in
% the image of $\varphi_\lambda:\cT(S)\to\cH(\lambda,\R)$ is that
% $$\Omega_{\text{Th}}(\sigma,\mu)>0,\quad\text{for every non-zero transverse
% measure }\mu.$$
% This turns out to be also suficient, that is, if $\sigma\in\cH(\lambda,\R)$ is a
% transverse cocycle satisfying the above condition then it is the transverse
% cocycle for some hyperbolic metric $m$, $\sigma=\varphi_\lambda(m)$.\footnote{This
% is another non-trivial statement for which we need to give a reference.}

The main reason we consider Thurston's intersection form is due to its relation with
the Weil-Petersson symplectic form.

\begin{theorem}[Bonahon-S\"ozen \cite{sozen-bonahon}] 
The map $\varphi_\lambda:(\cT,\omega_{WP})\to(\cH(\lambda,\R),\Omega_{\text{Th}})$
is a symplectomorphism $$\varphi_\lambda^*\Omega_{\text{Th}}=\omega_{WP}.$$
\end{theorem}

Similarly, the canonical cotangent bundle symplectic structure on $T^*\cT$ can also be related with Thurston's
intersection form. First, note that the map 
$\varphi_\lambda:\cT(S)\to\cH(\lambda,\R)$
naturally identifies the cotangent space to $\cT(S)$ at $m$ with 
the cotangent space to $\cH(\lambda,\R)$ at $\sigma_m$ which, furthermore, is just the dual space 
$\cH(\lambda,\R)^*$ to $\cH(\lambda,\R)$:
$$T^*_m\cT(S)=T^*_{\sigma_m}\cH(\lambda,\R)=\cH(\lambda,\R)^*.$$

The total space of the cotangent bundle $T^*\cT(S)$ over $\cT(S)$ 
is then identified with a subset
of $\cH(\lambda,\R)\times\cH(\lambda,\R)^*$ by
$$(\varphi_\lambda,(\varphi_\lambda^{-1})^*):(m,u)\mapsto
(\varphi_\lambda(m),(\varphi_\lambda^{-1})^*u)=(\sigma_m,\sigma_u^*).$$
Using the Thurston intersection form we may further identify the dual 
space $\cH(\lambda,\R)^*$ with $\cH(\lambda,\R)$ via
$$\sigma\mapsto\sigma^*=\Omega_{\text{Th}}(\;\cdot\;,\sigma)$$
so the symplectic form on $\cH(\lambda,\R)\times 
\cH(\lambda,\R)^*$ can be written as
$$\Omega_*\Big((\sigma_1,\tau^*_1),(\sigma_2,\tau^*_2)\Big)=\Omega_{\text
{Th}}(\tau_1,\sigma_2)-\Omega_{\text{Th}}(\tau_2,\sigma_1).$$

\begin{prop}
The map 
$(\varphi_\lambda,(\varphi_\lambda^{-1})^*):(T^*\cT,\omega_*^r)\to(\cH(\lambda,
\R)\times \cH(\lambda,
\R)^*,\Omega_*)$
is a symplectomorphism
$$ (\varphi_\lambda,(\varphi_\lambda^{-1})^*)^*\Omega_*=\omega_*^r~.$$
\end{prop}

\begin{proof}
We only need to compare the canonical Liouville 
1-forms $\theta$ on $T^*\cT(S)$ and $\Theta$ on $\cH(\lambda,\R)\times 
\cH(\lambda,\R)^*$
$$\theta_{(m,u)}(U)=u(\pi_*U),\quad \Theta_{(\sigma,\tau^*)}(\rho,
\chi^*)=\tau^*(\rho).$$

Pulling-back $\Theta$ by $(\varphi_\lambda,(\varphi_\lambda^{-1})^*)$ gives
$$((\varphi_\lambda,(\varphi_\lambda^{-1})^*)^*\Theta)_{(m,u)}(U) 
% =\Theta_{(\sigma,\tau^*)} 
% (\varphi_\lambda,(\varphi_\lambda^{-1})^*)_*(U) 
% =\tau^*((\varphi_\lambda)_*\pi_*U) 
=(\varphi_\lambda)^*u((\varphi_\lambda)_*\pi_*U)=\theta_{(m,u)}(U).$$
Thus
$$ 
(\varphi_\lambda,(\varphi_\lambda^{-1})^*)^*\Omega_*=(\varphi_\lambda,
(\varphi_\lambda^{-1})^*)^*d\Theta=d\theta=\omega_*^r~. $$ 
\end{proof}
%%\footnote{Have to check the minus sign conventions through out the paper}

\subsection{The double earthquake map is symplectic}

We now provide a proof that the double earthquake map $\cE'$ is symplectic, up
to a multiplicative factor, Theorem \ref{tm:double-earthquake}.
First we need a description of earthquakes along measured laminations in terms
of transverse cocycles for maximal laminations.

Thus, given $(m,l)\in\cT\times\cML$ let $m'=E_L(m,l)$ denote the left earthquake
of $m$ along $l$ and let $\lambda$ be a maximal lamination on $S$ containing the
support of $l$. Denote by $\sigma=\sigma_m$ the transverse cocycles
associated with $m$ and by $\tau$ the transverse measure of $l$. We now
compute the transverse cocycle $\sigma'=\sigma_{m'}$ corresponding to $m'$. Let
us fix a transverse arc $k$ to $\lambda$. Let $\tilde k$ be a lift of $k$ to the
universal cover of $S$. By transversality, the end points of $\tilde k$ lay in
the interior of triangles $P$, $Q$ in the triangulation of $\tilde S$ determined
by the complement $\tilde S\backslash\tilde\lambda$ of the preimage
$\tilde\lambda$ of $\lambda$. We only need to consider the case where $P$ and
$Q$ are adjacent since for non-adjacent triangles the cocycles are obtained as
the sum of cocycles of the intermediate pairs of triangles. The construction of
the transverse cocycle associated with a hyperbolic metric is given by 
orthogonally projecting the third vertex of $P$ and $Q$ to their common edge and
computing the signed hyperbolic distance between the obtained pair of points
(equivalently, this is given by the logarithm of the cross-ratio of the ideal
square determined by $P$ and $Q$). The action of the earthquake $E_L(l)$, as
viewed from $P$, is then to shift the projected point from $Q$ by $\tau$.
Therefore, the transformation of the $PQ$-cocycle is
$$\sigma_{PQ}\mapsto\sigma'_{PQ}=\sigma_{PQ}+\tau_{PQ}$$
where $\tau_{PQ}$ is the measure of any arc transversally intersecting
$\tilde\lambda$ a unique time at the common edge of $P$ and $Q$. If $P$ and $Q$
are non-adjacent, the formula
$$\sigma_{PQ}\mapsto\sigma'_{PQ}=\sigma_{PQ}+\tau_{PQ}$$
is still valid, where now $\sigma_{PQ},\tau_{PQ}$ are given by the sum (possibly with
an infinite number of terms) over intermediate pairs of triangles.
The measure of the transverse arc $k$ is then given by
$$\sigma_{m'}(k)=\sigma_m(k)+\tau(k)$$
and we see that the transverse cocycles of $m$ and $m'$ are related by
$$\sigma_{m'}=\sigma_m+\tau.$$

\begin{proof}[Proof of Theorem \ref{tm:double-earthquake}]
From the discussion above, we may write the double earthquake map 
$\cE:\cT\times\cML\to\cTT$ in
terms of transverse cocycles for $\lambda$ as
$$\cE_\lambda(\sigma,\tau)=(\varphi_\lambda,\varphi_\lambda)\circ\cE\circ(\varphi_\lambda^{-1},
\iota_\lambda)(\sigma,\tau)=(\sigma+\tau,
\sigma-\tau).$$
àHere we denote by $\iota_\lambda:\cH(\lambda,\R_+)\to\cML$ the map assigning to
a non-negative transverse cocycle $\tau$ the measured lamination with support $\lambda$
and transverse measure $\tau$.

On the other hand by the relation between the $m$-length of measured laminations
and Thurston's intersection form recalled above,
$$ L_m(l)=\Omega_{\text{Th}}(\sigma_m,\sigma_l)~, $$
we may also describe the inverse of the map $\delta:\cT\times\cML\to T^*\cT$ in
terms of cocycles by
$$ \delta_\lambda^{-1}(\sigma,\tau^*)=(\varphi_\lambda,\iota_\lambda)\circ\delta^{-1}\circ(\varphi_\lambda,
\varphi_\lambda^*)(\sigma,\tau^*)=(\varphi_\lambda^{-1}\sigma,\varphi_\lambda^*\tau^*)=(\sigma,
\tau)~. $$
Thus the double earthquake map $\cE':T^*\cT\to\cTT$ can be realized by
$$ \cE'_\lambda(\sigma,\tau^*)=\cE_\lambda\circ\delta_\lambda^{-1}(\sigma,\tau^*)=(\sigma+\tau,\sigma-\tau)~. $$

Now note that the map $\cE'_\lambda:\cH(\lambda,\R)\times\cH(\lambda,\R)^*\to\cH(\lambda,\R)\times\cH(\lambda,\R)$
defined above is a symplectomorphism (up to a multiplicative factor) with respect to the
cotangent bundle symplectic form on $\cH(\lambda,\R)\times\cH(\lambda,\R)^*$ and the
difference of Thurston intersection forms on $\cH(\lambda,\R)\times \cH(\lambda,\R)$
\begin{align*}
\cE'_\lambda{}^*(\Omega_{\text{Th}}\oplus\overline{\Omega_{\text{Th}}})\Big((\rho_1,\theta^*_1),(\rho_2,\theta^*_2)\Big)=\Omega_{\text{Th}}(\rho_1+\theta_1,\rho_2+\theta_2)-\Omega_{\text{Th}}
(\rho_1-\theta_1,\rho_2-\theta_2)
\cr
=2\Omega_{\text{Th}}(\theta_1,\rho_2)-2\Omega_{\text{Th}}(\theta_2,\rho_1)=2\Omega_{T^*\cT}\Big((\rho_1,\theta^*_1),(\rho_2,\theta^*_2)\Big)~.
\end{align*}

Finally, restricting to the appropriate subsets, we have
\begin{align*}
\cE'{}^*(\omom)=\cE'{}^*\circ(\varphi_\lambda^*,\varphi_\lambda^*)(\Omega_{\text
{Th}}\oplus\overline{\Omega_{\text{Th}}})
% \cr
% =((\varphi_\lambda,\varphi_\lambda)\circ\cE')^*(\Omega_{\text{Th}}
% \oplus\overline{\Omega_{\text{Th}}})
% \cr
% =(\cE'_\lambda\circ
% (\varphi_\lambda^{-1},(\varphi_\lambda^{-1})^*))^*(\Omega_{\text{Th}}
% \oplus\overline{\Omega_{\text{Th}}})
\cr
=(\varphi_\lambda^{-1},(\varphi_\lambda^{-1})^*)^*\circ
\cE'_\lambda{}^*(\Omega_{\text{Th}}\oplus\overline{\Omega_{\text{Th}}})
\cr
=(\varphi_\lambda^{-1},(\varphi_\lambda^{-1})^*)^*\circ
2\Omega_{T^*\cT}=2\omega_*^r~. 
\end{align*}
\end{proof}

\begin{proof}[Proof of Theorem \ref{tm:wick-cc}]
The proof that $W_\partial^{AdS}:\cHE\to\cGH_{-1}$ is symplectic now follows 
from Theorem \ref{tm:double-earthquake} and Remark \ref{rk:equiv-partial}.
\end{proof}

\subsection{The dual Schl\"afli formula for convex cores of AdS manifolds}
\label{ssc:dual-schlafli}

The main point of this section is a result on the variation, under a deformation, 
of the volume (or rather the dual volume) of the convex core of a globally
hyperbolic AdS manifold. Although not obviously related to the main results
of this paper, this formula is the key tool in proving, in the next section,
that the double harmonic map is symplectic.

The result presented here should be compared with the Schl\"afli formula
obtained by Bonahon \cite{bonahon:schlafli} for convex cores of quasifuchsian
hyperbolic manifolds, and to the dual formula, for the variation of the dual
volume of quasifuchsian manifolds, used in \cite{cp}. The result we prove here
(and need below) is the AdS analog of the dual Schl\"afli formula of \cite{cp}.
We do not consider here the Schl\"afli formula itself for AdS convex cores, however it 
is possible that it could be obtained from the dual formula by a fairly direct argument
(possibly similar to the argument used in the other direction in \cite{cp}
in the hyperbolic setting).

\begin{defi}
Let $g\in \cGH_{-1}$ be a GHM AdS metric on $M=S\times \R$. We denote
by $V_+(g)$ the volume of the domain of $M$ bounded by the unique maximal
surface $S\subset M$ and by the upper boundary $\partial_+C(M,g)$ of the
convex core of $M$, and set
$$ V_+^*(g)=V_+(g)-\frac 12L_m(l)~, $$
where $m$ and $l$ are the induced metric and the measured bending lamination 
on $\partial_+C(M,g)$.
\end{defi}

A key point of the proof of the symplecticity of the double harmonic map 
%\ref{pr:harmonic}
will be the following variation formula for the volume $V_+^*$.

\begin{lemma} \label{lm:variation}
The function $V_+^*:\cGH_{-1}\to \R$ is tangentiable.
In a first-order variation of the GHM AdS metric on $M$, the first-order
variation of $V_+^*$ is 
\begin{equation}
  \label{eq:schlafli-half}
   (V_+^*)' = -\frac 14\int_S \langle I',\II\rangle da_I 
- \frac 12 d_mL(l)(m')~. 
\end{equation}
\end{lemma}

Here $I$ and $\II$ are the induced metric and second fundamental form on the
unique maximal Cauchy surface $S$ in $M$.

The dual Schl\"afli formula for convex core of GHM AdS manifolds follows directly.

\begin{prop} \label{pr:schlafli}
In a first-order variation of the GHM AdS metric $g$, 
$$ (V^*(g))' = -\frac 12 d_mL(l)(m')~, $$
where the left-hand side now includes both the upper and lower boundary components
of the convex core.
\end{prop}

\begin{proof}
This follows directly from applying \ref{lm:variation} both to $V_+^*$ and to the 
corresponding quantity $V_-^*$ for the part of $M$ between the maximal surface 
$S$ and the lower boundary of the convex core, that is, the quantity corresponding
to $V^*_+$ after changing the time orientation of $M$. The first term on the
right-hand side of \eqref{eq:schlafli-half} is then exactly compensated by the
corresponding term for the lower half of the convex core, and only the second
term remains.
\end{proof}

The proof of Lemma \ref{lm:variation} will basically follow from a first variation formula 
for the volume of AdS domains with smooth boundary. In the following statement
we denote by $I,\II,H$ the induced metric, second fundamental form and
mean curvature of the boundary, with $H=tr_I(\II)$, and suppose that the 
orientation conventions are such that $\II$ is positive when the boundary is convex.

\begin{lemma} \label{lm:schlafli}
Let $\Omega$ be a 3-dimensional manifold with boundary, with a one-parameter
family of AdS metrics $(g_t)_{t\in [0,1]}$ such that the boundaryproof is smooth 
and
space-like. Then 
$$ 2V(\Omega)'= \int_{\partial \Omega} H' + \frac 12\langle I',\II\rangle da_I~. $$
Here $V(\Omega)'=(d/dt)V(\Omega,g_t)_{|t=0}$ and similarly for the other primes. 
\end{lemma}

%%\begin{proof}[Sketch of the proof]
This statement is the exact Lorentzian analog, in the 3-dimensional case, 
of \cite[Theorem 1]{sem-era} (see also \cite{sem} for a complete proof).
The argument there can be used almost with no modification here. We leave the
details to the interested reader.
%% \end{proof}

Note that Lemma \ref{lm:schlafli} could be stated in a much more general way
by considering a higher-dimensional manifold with a one-parameter family of
Einstein metrics, as in \cite{sem}. The fact that the boundary is space-like
is not essential. Note also that an alternate proof can be found, for Riemannian
Einstein manifolds, in \cite{renormvol}.

\begin{cor}\label{cr:schlafli}
Under the same conditions as in Lemma \ref{lm:schlafli}, let 
$$ V^*(\Omega) = V(\Omega) - \frac 12\int_{\partial \Omega} Hda_I~. $$
Then 
$$ 2V^*(\Omega)'= \int_{\partial \Omega} \frac 12\langle I',\II-HI\rangle_I da_I~. $$
\end{cor}

\begin{proof}
This follows from Lemma  \ref{lm:schlafli} because an elementary computation
shows that
$$ \left( \int_{\partial\Omega} Hda_I\right)' = \int_{\partial\Omega} H' + 
\frac H2\langle I',I\rangle_I da_I~. $$
\end{proof}

The last technical tool that will be needed in the proof of Lemma \ref{lm:variation}
is the description of the surfaces equidistant from a convex pleated surface
in $AdS^3$. This description is directly analogous to what is well-known for
the equidistant surfaces from a convex pleated surface in $\bbH^3$, so we give
only a brief account here, leaving the details to the reader.
We consider a past-convex space-like pleated surface $\Sigma\subset AdS^3$, denote its
induced metric by $m$ and its measured pleating lamination by $l$, and will
denote by $\Sigma_r$ the equidistant surface at time-distance $r$ in the
past of $\Sigma$ (i.e., in the convex domain bounded by $\Sigma$ --- this
contrasts with the hyperbolic situation where one typically considers 
the equidistant surface in the concave region).

The simplest case occurs when $l=0$ and $\Sigma$ is totally geodesic. Then
a simple computation shows that $\Sigma_r$ is umbilic and future-convex, with principal 
curvatures equal to $-\tan(r)$. If on the other hand we suppose that $\Sigma$ 
is made of two
totally geodesic half-planes $P_1$ and $P_2$ intersecting at an angle $\theta$
along their common
boundary, we obtain that $\Sigma_r$ has three components:
\begin{itemize}
\item two umbilic surfaces $P_{1,r}$ and $P_{2,r}$, with orthogonal projection
on $\Sigma$ respectively on $P_1$ and $P_2$,
\item a strip $S$ of width $\theta\sin(r)$, which projects orthogonally to 
$\partial P_1= \partial P_2$, 
where one principal direction (along the axis) is $1/\tan(r)$, while the other
is $-\tan(r)$.
\end{itemize}
Suppose now that $\Sigma$ is a past-convex space-like pleated surface in a GHM AdS manifold,
with rational measured bending lamination $l$. It follows from the previous description that
$\Sigma_r$ has umbilic regions (projecting orthogonally to the complement of the support of
$l$ in $\Sigma$) with principal curvatures $-\tan(r)$, and ``strips'' projecting orthogonally
to the support of $l$, with principal curvatures equal to $1/\tan(r)$ and to $-\tan(r)$.
In particular, it will be important below to note that the area of $\Sigma_r$ is
$$ A_r(\Sigma_r) = \cos^2(r)(-2\chi(S)) + \sin(r)\cos(r) L_m(l)~. $$
It follows by continuity that the same area formula holds for general (not rational)
measured bending lamination.

We can now provide a direct proof of Lemma \ref{lm:variation}. 
Note that this contrasts with the argument given in \cite{cp}, where the 
``dual Schl\"afli formula'' was proved using Bonahon's Schl\"afli formula
(see \cite{bonahon-variations,bonahon}). It appears likely that, in 
the hyperbolic setting too, a direct proof of the dual Schl\"afli formula
can be given without going through Bonahon's Schl\"afli formula, which 
is more complicated even to state since it involves the first-order
variation of the measured bending lamination.

\begin{proof}[Proof of Lemma \ref{lm:variation}]
We will consider a smooth one-parameter family $(g_t)_{t\in [0,1]}$ of AdS 
metrics on $M$, and let $(m_t)_{t\in [0,1]}$ and $(l_t)_{t\in [0,1]}$ be the
induced metric and measured bending lamination of the upper boundary
of the convex core. 

Let $\dot g=(dg_t/dt)_{|t=0}$, then $\dot g$ determines a first-order variation
$(\dot m_L,\dot m_R)$ of the left and right metrics $(m_L, m_L)$ of $g$. 
So Corollary \ref{cr:EC1} shows that $\dot g$ determines a tangent vector 
$d(\cE\circ \delta^{-1})(\dot m_L,\dot m_R)\in T_{\cE\circ \delta^{-1}(m_L,m_R)}T^*\cT$.

Since $\delta:\cT\times \cML\to T^*\cT$ is tangentiable (as proved in \cite{bonahon-toulouse}), 
it follows that $\dot g$ determines a first-order variation  
$\dot m\in \cT_m\cT$ of $m$, and a first-order deformation $\dot l\in \cH(\lambda, \R)$
of $l$.
The first-order variation $\dot l$ should be understood as a transverse cocycle,
$\dot l\in \cH(\lambda,\R)$, as outlined in Section \ref{ssc:maxlam}, where $\lambda$
is a maximal lamination containing the support of $l$. We can now slightly change the
perspective and consider $(m,l)$ as the main variables, and a first-order
variation $\dot g$ of $g$ determined by first-order variations $\dot m$ and $\dot l$
of $m$ and $l$.

Following the definitions above, we denote by $M(m,l)$ the GHM AdS
manifold such that the upper boundary of the convex core has induced metric
$m$ and measured bending lamination $l$, by $S(m,l)$ the unique maximal
Cauchy surface in $M(m,l)$, by $\Omega(m,l)$ the domain in 
$M(m,l)$ bounded by $S(m,l)$ and by the upper boundary of the convex core
$\partial_+\Omega(m,l)$, and by $V_+(m,l)$ its volume. We then set
$$ V^*_+(m,l)=V_+(m,l)-\frac 12 L_m(l)~. $$
We now have to prove that this function 
is tangentiable with the correct derivative.

Our strategy to prove the variation formula for $V^*_+$ will be to approximate
the pleated surface $\partial_+\Omega(m,l)$ by equidistant surfaces, to which
we can apply the smooth dual Schl\"afli formula of Lemma \ref{lm:schlafli}.
So, for $r>0$, we denote by $\Sigma_r$ the set of points at 
time distance $r$ from $\partial_+\Omega(m,l)$ in the past. If $r$ is
small enough, then $\Sigma_r\subset \Omega(m,l)$. We then call 
$\Omega_r(m,l)$ the compact domain in $M(m,l)$ bounded by $S(m,l)$ and
$\Sigma_r(m,l)$. So $\Omega_r(m,l)$ is contained in $\Omega(m,l)$, more
precisely $\Omega(m,l)$ is composed of all points at time distance at most 
$r$ from $\Omega_r(m,l)$ in its future.

We denote by $I_r, \II_r, H_r$ the induced metric, second fundamental
form and mean curvature of $\Sigma_r$, and define 
$$ V^*_r(m,l) = V(\Omega_r(m,l)) - \frac 12\int_{\Sigma_r} H_r da_{I_r}~. $$

The first-order variation formula for $V^*_r$ follows 
from Lemma \ref{lm:schlafli} and the proof of Corollary \ref{cr:schlafli}:
%%\footnote{We have to check the signs and orientation conventions for
%%$S$ and $\Sigma_r$.}
\begin{equation}
  \label{eq:var-Vr}
(V^*_r(m,l))' = -\frac 12 \int_S H'+\frac 12\langle I',\II\rangle da_I 
+ \frac 14\int_{\Sigma_r} \langle I'_r,\II_r-H_rI_r\rangle_{I_r} da_{I_r}~. 
\end{equation}
Note that the terms corresponding to $\Sigma_r$ is different from the
term on $S$ since, in the definition of $V^*_r(m,l)$, an integral mean curvature term is 
added but it is only an integral on $\Sigma_r$.
The first integral already occurs in the statement of Lemma \ref{lm:variation},
and moreover $H'=0$ since $S$ remains a maximal surface throughout the deformation.
So, to prove the statement, we need to show that
\begin{equation}
  \label{eq:goal}
\int_{\Sigma_r} \langle I'_r,\II_r-H_rI_r\rangle_{I_r} da_{I_r}
\stackrel{r\to 0}{\longrightarrow} -2d_mL(l)(\dot m)~.
\end{equation}

For $r>0$ small enough, $\Sigma_r$ is $C^{1,1}$ smooth --- this is the 
Lorentzian analog of the well-known fact that the equidistant surface
from a convex pleated surface in hyperbolic space, on the concave side
of the complement, is $C^{1,1}$ smooth. Note that $\Sigma_r$ is not convex, but this will
not play any role in the argument.

There is a well-defined nearest-point projection 
$\rho:\Sigma_r(m,l)\to \partial_+\Omega(m,l)$. Therefore we can decompose
$\Sigma_r(m,l)$ in two components:
\begin{itemize}
\item $\Sigma_r^l(m,l)$ is the inverse image by $\rho$ of the support of
$l$, so that is a closed subset of $\Sigma_r(m,l)$,
\item $\Sigma_l^f(m,l)=\Sigma_r(m,l)\setminus \Sigma_r^l$ is the open set of 
points which project to a point of $\Sigma(m,l)$ which has a totally
geodesic neighborhood.
\end{itemize}
Both $\Sigma_r^l(m,l)$ and $\Sigma_r^f(m,l)$ are smooth surfaces.

The area of $\Sigma_r^f$ depends on the area of $\partial_+C(M)$, 
specifically:
$$ A_r(\Sigma_r^f) = \cos^2(r) (-2\pi \chi(S))~. $$
Similarly, the area of $\Sigma_r^l$ depends on the length of $l$ for $m$:
$$ A_r(\Sigma_r^l) = \sin(r)\cos(r) L_m(l)~. $$
As a consequence, we can express the volume of $\Omega_r(m,l)$ in 
terms of the volume of $\Omega_+(m,l)$:
$$ V(\Omega(m,l)) - V_+(m,l) = \int_{s=0}^r \sin(s)\cos(s) L_m(l)
+ \cos^2(s) (-2\pi \chi(S)) ds \stackrel{r\to 0}{\longrightarrow} 0~. $$
Moreover
$$ \int_{\Sigma_r(m,l)} H da_I \stackrel{r\to 0}{\longrightarrow} 
L_m(l)~, $$
and it follows that $V^*_r(m,l)\to V^*(m,l)$ in the local $C^0$ sense
as $r\to 0$.

Clearly, $\Sigma_r^f(m,l)$ is the disjoint union of open surfaces which
are equidistant from a plane and therefore umbilic, 
with principal curvatures equal to $-\tan(r)$. The local
geometry of $\Sigma_r^l(m,l)$ is slightly more interesting. It has a  
foliation $\Lambda$ by geodesics, each of which project to a leave of $l$.
The directions parallel to $\Lambda$ are principal directions, with
corresponding principal curvature $-\tan(r)$, while the principal
curvature corresponding to the directions orthogonal to $\Lambda$
is $\cotan(r)$. 

As a consequence, the mean curvature of $\Sigma_r$ is equal to
$-\tan(r)+\cotan(r)$ on $\Sigma_r^l$, and to $-2\tan(r)$ on $\Sigma_r^f$. 
It follows that $\II_r-H_rI_r$ is equal to 
\begin{itemize}
\item $-\cotan(r)I_r$ on directions
parallel to $\Lambda$ on $\Sigma_r^l$, 
\item $\tan(r)I_r$ on directions
orthogonal to $\Lambda$ on $\Sigma_r^l$ and on all directions in $\Sigma_r^f$.
\end{itemize}

To prove \eqref{eq:goal}, we decompose the first-order variation of 
$I_r$ in two terms: $dI_r(\dot m)$ corresponding to $\dot m$, and $dI_r(\dot l)$ 
corresponding to $\dot l$. We will compute separately the contribution
of each term to the limit of the integral on the left-hand side of
\eqref{eq:goal}. 
For both computations, we will consider the area $A_r=A_r(\Sigma_r^l)+A_r(\Sigma_r^f)$ 
of $I_r$. Similarly as in the hyperbolic setting (see eg \cite{minsurf}) we have
$$ A_r = -2\pi\chi(S) + \sin(r)\cos(r) L_m(l)~. $$

Note that the first-order deformation $dI_r(\dot l)$ of $I_r$ vanishes in the directions
parallel to $\Lambda$ on $\Sigma_r^l$. So it follows from the description of
$\II_r-H_rI_r$ given above that
\begin{eqnarray*}
\int_{\Sigma_r} \langle dI_r(\dot l),\II_r-H_rI_r\rangle_{I_r} da_{I_r}
& = & \int_{\Sigma_r} \langle dI_r(\dot l),\tan(r) I_r\rangle_{I_r} da_{I_r} \\
& = & 2\tan(r) dA_r(\dot l) \\
& = & 2\sin^2(r) dL_m(\dot l) \\
& = & 2\sin^2(r) L_m(\dot l) \\
&  \stackrel{r\to 0}{\longrightarrow} & 0~.
\end{eqnarray*}

Similarly, $dI_r(\dot m)$ is bounded on $\Sigma_r^f$, while it vanishes
on $\Sigma_r^l$ on directions orthogonal to $\Lambda$. It follows that 
$$ \int_{\Sigma_r} \langle dI_r(\dot m),\II_r-H_rI_r\rangle_{I_r} da_{I_r}
= \int_{\Sigma_r^l} \langle dI_r(\dot m),-\cotan(r)I_r\rangle_{I_r} da_{I_r}
+ \int_{\Sigma_r^f} \langle dI_r(\dot m),\tan(r)I_r\rangle_{I_r} da_{I_r}~. $$
However 
$$ \int_{\Sigma_r^f} \langle dI_r(\dot m),\tan(r)I_r\rangle_{I_r} da_{I_r}
\stackrel{r\to 0}{\longrightarrow} 0~, $$
while
\begin{eqnarray*}
\int_{\Sigma_r^l} \langle dI_r(\dot m),-\cotan(r)I_r\rangle_{I_r} da_{I_r}
& = & 
-\cotan(r)\int_{\Sigma_r^l} \langle dI_r(\dot m),I_r\rangle_{I_r} da_{I_r} \\
& = & 
-2\cotan(r) dA_r(\Sigma_r^l)(\dot m) \\
& = & 
-2\cos^2(r) d_mL(l)(\dot m) \\
&  \stackrel{r\to 0}{\longrightarrow} & 
-2d_mL(l)(\dot m)~.
\end{eqnarray*}
Summing up, we obtain Equation \eqref{eq:goal}. 

Therefore, 
$$ dV^*_r(m,l)\to -\frac 12\int_S H'+\frac 12\langle I',\II\rangle da_I 
- \frac 12 d_mL(l)(m')~. $$ 
pointwise as $r\to 0$.
Since $V^*_r(m,l)\to V^*(m,l)$ in $C^0$ as $r\to 0$, the result
follows. 
\end{proof}

\subsection{The double harmonic map is symplectic}

We turn here to the proof of Theorem \ref{tm:harmonic}: the double harmonic
map $\cH:T^*\cT\to \cTT$ is symplectic up to a factor, more precisely,
$$ \cH^*(\omom) = -\omega^r_*~, $$
where $\omega^r_*$ is the real part of the complex symplectic structure on $T^*\cT$.

The key part of the argument is the dual Schl\"afli formula, more specifically
Lemma \ref{lm:variation} seen in the previous section. Note that a similar argument
was used in the hyperbolic setting by Loustau in \cite{loustau:minimal}.

We will use the diagram in Figure \ref{fg:harmonic}, which is a variant of other similar
(related) diagrams presented in the paper.

\begin{figure}[h] 
\begin{center}
\includegraphics[width=0.43\textwidth]{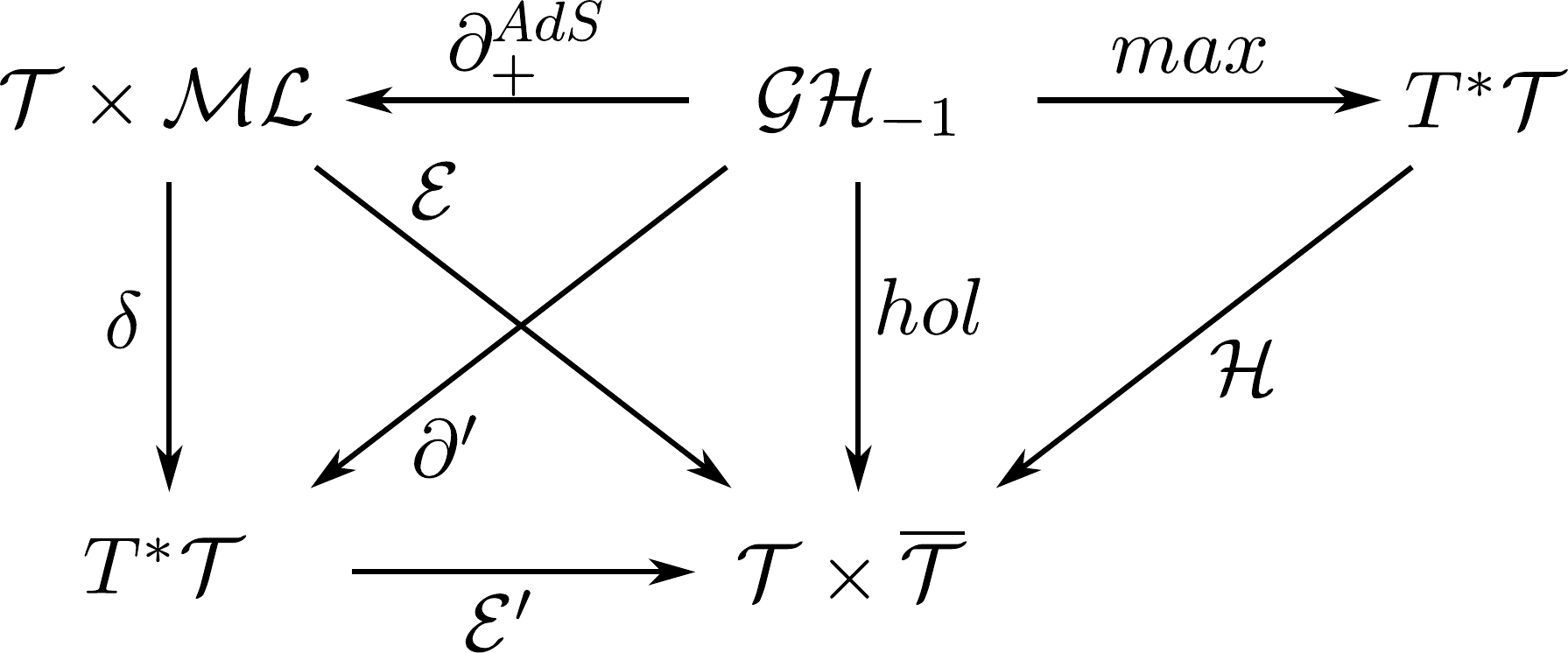}
\caption{Earthquakes and harmonic maps associated to GHM AdS manifolds}
\label{fg:harmonic}
\end{center}
\end{figure}

% \begin{figure}[h] 
% \begin{center}
% %\begin{small}
% \begin{tikzpicture}[node distance=3cm, auto]
%   \node (GH) {$\cGH_{-1}$};
%   \node (T*T1) [left of=GH]{$T^*\cT$};
%   \node (T*T2) [right of=GH]{$T^*\cT$};
% %  \node (T*T3) [below of=T*T2]{$T^*\cT$};
%   \node (TML) [below of=T*T1]{$\cT\times \cML$};
%   \node (TT) [below of=GH]{$\cTT$};
%   \draw[->] (GH) to node {$\partial'$} (T*T1);
%   \draw[->] (GH) to node[above] {$max$} (T*T2);
%   \draw[->] (TML) to node{$\cE$} (TT);
%   \draw[->] (T*T2) to node{$\cH$} (TT);
%   \draw[->] (TML) to node {$\delta$} (T*T1);
%   \draw[->] (GH) to node {$hol$} (TT);
%   \draw[->] (GH) to node[below left] {$\partial_+^{ads}$} (TML);
%   \draw[->] (T*T1) to node[above left] {$\cE'$} (TT);
% \end{tikzpicture}
% %\end{small}
% \caption{Earthquakes and harmonic maps associated to GHM AdS manifolds}
% \label{fg:harmonic}
% \end{center}
% \end{figure}

In this diagram we denote by $\partial':\cGH_{-1}\to T^*\cT$ the composition 
$\partial'=\delta\circ\partial_+^{AdS}$. 

This diagram is commutative. The fact that the right triangle commutes is
a direct translation of Lemma \ref{lm:max-hopf}. In the left square, 
the triangles not involving the $hol$ map commute by definition, while
the two triangles involving $hol$ commutes by Theorem \ref{tm:mess-diagram}
and Lemma \ref{lm:max-hopf}.

\begin{prop} \label{pr:harmonic}
The map $\maxi\circ {\partial'}^{-1}$ is symplectic up to a factor $-2$: 
$(max\circ {\partial'}^{-1})^*\omega_*^r=-2\omega_*^r$.
% %\footnote{We should check the coefficient, in particular the sign, here too!}
\end{prop}

\begin{proof}[Proof of Proposition \ref{pr:harmonic}]
Recall that the map $\delta:\cTML\to T^*\cT$ is defined as 
$\delta(m,l)=d_mL(l)$. 
Let $\theta$ denote the canonical Liouville 1-form of $T^*\cT$, that is, the 
1-form on $T^*\cT$
defined at a point $(m,u)\in T^*\cT$ by 
$$ \forall U\in T_{(m,u)}T^*\cT,\quad  \theta(U)=u(\pi_*U)~, $$
where $\pi:T^*\cT\to \cT$ is the canonical projection.
It follows from the defintion of $\delta$ that 
$$ \delta^*\theta(\dot m, \dot l) = d(L(l))(\dot m)~. $$
Pulling back this 1-form on $\cGH_{-1}$ by the map $\partial'$, we obtain that
$$ ((\partial')^*\theta)(\dot m,\dot l) = ((\partial \circ 
\delta)^*\theta)(\dot m,\dot l) 
= d(L(l))(\dot m)~, $$
where $(\dot m, \dot l)$ is now taken to define a tangent vector to $\cGH_{-1}$, as seen
at the beginning of Section \ref{ssc:dual-schlafli}.

A very similar argument shows that 
$$ (\maxi^*\theta)(\dot I,\dot \II) = \int_S \langle \dot I,\II\rangle da_I~, $$
where $(I,\II)$ determine a point in $\cGH_{-1}$ and $(\dot I,\dot \II)$ a tangent
vector to $\cGH_{-1}$ at this point.

Lemma \ref{lm:variation} can therefore be stated as follows: on $\cGH_{-1}$, 
$$ dV_+^* = -\frac 14 \maxi^*\theta - \frac 12 (\partial')^*\theta~. $$
Taking the differential, we obtain that 
$$ 0 = -\frac 14 \maxi^*\omega_*^r - \frac 12 (\partial')^*\omega_*^r~, $$
and therefore that 
$$ (\maxi\circ (\partial')^{-1})^*\omega_*^r+2\omega_*^r=0~. $$
\end{proof}

\begin{proof}[Proof of Theorem \ref{tm:harmonic}]
The proof clearly follows from Proposition \ref{pr:harmonic}, and from the diagram in
Figure \ref{fg:harmonic}, because Theorem \ref{tm:double-earthquake} asserts that
$$ (\cE')^*(\omom) = 2\omega^r_*~. $$
%%\footnote{Check the coefficients here too!}
\end{proof}

\begin{proof}[Proof of Theorem \ref{tm:minimal}]
The proof that the map $W_{min}:\cAF\to \cGH$ is symplectic follows from Theorem 
\ref{tm:harmonic} and from Remark \ref{rk:equiv2}.
\end{proof}

%%% Local Variables: 
%%% mode: latex
%%% TeX-master: "doublemaps"
%%% End: 

%% file: doublemaps6.tex
\section{Constant mean curvature surfaces} \label{sc:cmc}

In this section we consider the symplectic structures induced on the various
moduli spaces of geometric structures in 3 dimensions ($\cAF', \cGH_{-1}, \cGH_0$ and 
$\cGH_1$) by their identification with $T^*\cT$ through constant mean curvature surfaces.
We then prove Theorem \ref{tm:wick-cmc}.

\subsection{CMC surfaces in hyperbolic manifolds}

Recall that $\cAF'$ denotes the subspace of $\cAF$ of almost-Fuchsian metrics on 
$S\times \R$ which admit a foliation by CMC surfaces, with mean curvature going from
$-2$ to $2$. Conjecturally, $\cAF'=\cAF$. An elementary application of the maximum
principle shows that for $h\in \cAF'$, $(M,h)$ contains a unique closed, 
embedded
CMC-$H$ surface, which is a leave of the CMC foliation.

\begin{defi}
For all $H\in (-2,2)$, we denote by $\cmc^{Hyp}_H:\cAF'\to T^*\cT$ the map sending 
a hyperbolic metric $h\in \cAF'$ to $([I],\II_0)$, where $[I]$ is the conformal 
class
of the induced metric and $\II_0$ is the traceless part of the second fundamental
form of the unique closed, embedded CMC-$H$ surface in $(M,h)$.
\end{defi}

A key point for us is that the symplectic form obtained on $\cAF'$ by pulling back
the cotangent symplectic structure on $T^*\cT$ to $\cAF'$ by all those maps is always
the same. We will see below that the same result, basically with the same proof, 
extends to globally hyperbolic constant curvature space-times.

\begin{prop} \label{pr:same-hyp}
Let $H,H'\in (-2,2)$. Then $(\cmc^{Hyp}_H)^*\omega_*^r = 
(\cmc^{Hyp}_{H'})^*\omega_*^r$.
\end{prop}

\begin{proof}
We suppose, without loss of generality, that $H'>H$. Let $\Sigma$ and 
$\Sigma'$ be the closed, embedded surfaces with constant mean curvature 
$H$ and $H'$, respectively, and let $\Omega$ be the domain bounded by
$\Sigma$ and $\Sigma'$. We orient both $\Sigma$ and $\Sigma'$ towards increasing
values of $H$. We define
$$ V^*(\Omega) = Vol(\Omega) -\frac 12 \int_{\Sigma'} H'da_I + \frac 12 \int_{\Sigma} Hda_I~. $$
Corollary \ref{cr:schlafli} then indicates that, in a first-order deformation of $g$,
$$ 2V^*(\Omega)' = \int_{\Sigma'} \frac 12\langle I',\II-H'I\rangle_I da_I -
\int_\Sigma \frac 12\langle I',\II-HI\rangle_I da_I~. $$
(Note that the signs are slightly different from those in Corollary \ref{cr:schlafli} because
the orientation of $\Sigma$ is different, here it is towards increasing values of $H$
and therefore towards the interior of $\Omega$.)

Clearly we have
$$ \II = \II_0+\frac H2 I~, $$
so that 
$$ \II-HI = \II_0 - \frac H2 I~. $$
As a consequence,
\begin{eqnarray*}
2V^*(\Omega)' & = & \int_{\Sigma'} \frac 12\langle I',\II_0\rangle_I da_I -
\frac {H'}2 \int_{\Sigma'}\frac 12 \langle I',I\rangle_I da_I - 
\int_\Sigma \frac 12\langle I',\II_0\rangle_I da_I + 
\frac H2 \int_{\Sigma}\frac 12 \langle I',I\rangle_I da_I \\
& = & \frac 12\int_{\Sigma'} \langle I',\II_0\rangle_I da_I 
-\frac 12\int_{\Sigma} \langle I',\II_0\rangle_I da_I 
- \frac {H'}2 A(\Sigma')' + \frac H2 A(\Sigma)'~. 
\end{eqnarray*}
Another way to state this is that
$$ 2d\left(2V^*(\Omega) +\frac{H'}2A(\Sigma') -\frac H2A(\Sigma)\right) = 
(\cmc^{Hyp}_{H'})^*\theta - (\cmc^{Hyp}_{H})^*\theta~, $$
where $\theta$ is the Liouville form on $T^*\cT$. It follows that
$$ (\cmc^{Hyp}_{H'})^*\omega_*^r - (\cmc^{Hyp}_{H})^*\omega_*^r
= d((\cmc^{Hyp}_{H'})^*\theta - (\cmc^{Hyp}_{H})^*\theta) = 0~. $$
\end{proof}

\subsection{CMC surfaces in Lorentzian space-times}

Recall that, according to Theorem \ref{tm:ads:foliation}, any GHM AdS manifold
admits a unique foliation by CMC surfaces, with mean curvature going monotonically
from $-\infty$ to $\infty$. This makes the following definition possible.

\begin{defi}
For all $H\in \R$, we call $\cmc^{AdS}_H:\cGH_{-1}\to T^*\cT$ the map sending 
a GHM AdS metric $g\in \cGH_{-1}$ to $([I],\II_0)$, where $[I]$ is the conformal class
of the induced metric and $\II_0$ is the traceless part of the second fundamental
form of the unique closed, embedded CMC-$H$ surface in $(M,g)$.
\end{defi}

\begin{prop} \label{pr:same-ads}
Let $H,H'\in (-\infty,\infty)$. Then $(\cmc^{AdS}_H)^*\omega_*^r = 
(\cmc^{AdS}_{H'})^*\omega_*^r$.
\end{prop}

The proof is exactly the same as in the hyperbolic setting, since the dual Schl\"afli
formula has the same statement.

Things are similar in the de Sitter setting. According to Theorem \ref{tm:flat_ds:foliation},
any GHM de Sitter manifold has a unique foliation by CMC surfaces, with mean curvature 
varying between $-\infty$ and $-2$ (with the orientation conventions used here).

\begin{defi}
For all $H\in (-\infty, -2)$, we call $\cmc^{dS}_H:\cGH_{1}\to T^*\cT$ the map 
sending 
a GHM dS metric $g\in \cGH_{1}$ to $([I],\II_0)$, where $[I]$ is the conformal class
of the induced metric and $\II_0$ is the traceless part of the second fundamental
form of the unique closed, embedded CMC-$H$ surface in $(M,g)$.
\end{defi}

\begin{prop} \label{pr:same-ds}
Let $H,H'\in (-\infty,-2)$. Then $(\cmc^{dS}_H)^*\omega_*^r = 
(\cmc^{dS}_{H'})^*\omega_*^r$.
\end{prop}

The proof is again almost the same as for Proposition \ref{pr:same-hyp} above. The 
smooth Schl\"afli formula has a different sign in de Sitter manifolds, and it now
reads:
$$ 2V(\Omega)' = -\int_{\partial \Omega} H' + \frac 12\langle I',II\rangle_I da_I~. $$
Therefore one has to define the dual volume as
$$ V^*(\Omega) = V(\Omega) + \frac 12\int_{\partial\Omega} Hda_I~, $$
and the variation formula for $V^*$ has a minus sign compared to the hyperbolic or
AdS cases. However the proof of Proposition \ref{pr:same-ds} can be done as the
proof of Proposition \ref{pr:same-hyp}, with obvious sign differences.

Finally, in the Minkowski space, Theorem \ref{tm:flat_ds:foliation} indicates that
any GHM Minkowski manifold has a unique foliation by CMC surfaces, with mean
curvature varying between $-\infty$ and $0$. 

\begin{defi}
For all $H\in (-\infty, 0)$, we call $\cmc^{Mink}_H:\cGH_{0}\to T^*\cT$ the map 
sending 
a GHM AdS metric $g\in \cGH_{0}$ to $([I],\II_0)$, where $[I]$ is the conformal class
of the induced metric and $\II_0$ is the traceless part of the second fundamental
form of the unique closed, embedded CMC-$H$ surface in $(M,g)$.
\end{defi}

\begin{prop} \label{pr:same-mink}
Let $H,H'\in (-\infty,-2)$. Then $(\cmc^{Mink}_H)^*\omega_*^r = 
(\cmc^{Mink}_{H'})^*\omega_*^r$.
\end{prop}

The proof is again similar, but with larger differences. The smooth Schl\"afli
formula now reads as 
$$ \int_{\partial \Omega} H'+\frac 12\langle I',\II\rangle_I da_I = 0~. $$
We now define 
$$ \cH(\Omega) = \int_{\partial\Omega} Hda_I~, $$
and have the following variation formula for $\cH$ under a first-order deformation:
$$ \cH(\Omega)' = \int_{\partial\Omega} \frac 12 \langle I',\II-HI\rangle_Ida_I~. $$
The proof of Proposition \ref{pr:same-mink} can then proceed as the proof of 
Proposition \ref{pr:same-hyp}, with $\cH$ instead of $V^*$.

\begin{proof}[Proof of Theorem \ref{tm:wick-cmc}]
 
Note that for all $H\in (-2,2)$ and $H'\in (-\infty,\infty)$, we have
$$ W^{AdS}_{H,H'} = (\cmc^{AdS}_{H'})^{-1}\circ \cmc^{Hyp}_H~. $$

We first consider the special case where $H=H'=0$. With the notations
used above, $\cmc^{Hyp}_0=\min$ while $\cmc^{AdS}_0=\max$. We already know
by Theorem \ref{tm:loustau-min} that $\min:(\cAF,\omega_G^i)\to 
(T^*\cT,\omega_*^r)$ 
is symplectic up to the sign, that is 
$$ {\min}^*\omega_*^r = -\omega_G^i~. $$
Moreover, $W^{AdS}_{0,0}=W_{min}:(\cAF,\omega_G^i)\to (\cGH_{-1},\omom)$ 
is symplectic by Theorem \ref{tm:minimal}. It follows that 
$\max:(\cGH_{-1},\omom)\to (T^*\cT,\omega_*^r)$ is also symplectic up to sign.

Proposition \ref{pr:same-hyp} and Proposition \ref{pr:same-ads} therefore
indicate that for all $H\in (-2,2)$ and $H'\in (-\infty,\infty)$, 
$\cmc^{AdS}_{H'}$ and $\cmc^{Hyp}_H$ are also symplectic up to sign. 
Therefore, $W^{AdS}_{H,H'}:(\cAF',\omega_G^i)\to (\cGH_{-1},\omom)$ is also
symplectic.
\end{proof}

%%% Local Variables: 
%%% mode: latex
%%% TeX-master: "doublemaps"
%%% End: 

%% file: doublemaps7.tex
\section{Minkowski and de Sitter manifolds} \label{sc:mink-ds}

In this section we prove that the symplectic structure $\omega_G^i$ on the moduli space 
$\cGH_1$ of globally hyperbolic de Sitter manifolds is identical (up to the sign) to the symplectic
structure induced by the identification of $\cGH_1$ with $T^*\cT$ through CMC surfaces.
The proof of Theorem \ref{tm:wick-ds} will follow.

We then describe some conjectural statements for globally hyperbolic Minkowski manifolds.

\subsection{De Sitter CMC Wick rotation are symplectic}
\label{ssc:dS-CMC-wick}

The proof of Theorem \ref{tm:wick-ds} is mostly based, in addition to the content 
of the previous sections, on the following proposition. We call 
$\Delta:\cHE\to \cGH_1$
the duality map, that is, the map sending a hyperbolic end $E$ to the ``dual'' GHM
de Sitter manifold, which has the same complex projective structure at future
infinity as $E$. So $\Delta=(\partial_\infty^{dS})^{-1}\circ \partial_\infty^{Hyp}$ 
is a homeomorphism from $\cHE$ to $\cGH_1$, such that
$\Delta^*\omega_G^i=\omega_G^i$. We also call $\Delta'$ the restriction of 
$\Delta$ to the space $\cAF'$ of almost-Fuchsian metrics admitting a folation 
by CMC surfaces.

\begin{prop} \label{pr:duality}
For all $H_*\in (-\infty,-2)$ and all $H\in (-2,2)$, we have
$$ (\cmc^{dS}_{H_*}\circ \Delta')^*\omega_*^r = (\cmc^{Hyp}_H)^*\omega_*^r~. $$
\end{prop}

The proof is based on a basic differential geometry computation concerning the
term which appears in the smooth Schl\"afli formula of Lemma \ref{lm:schlafli}.

\begin{lemma} \label{lm:duality}
Let $\Sigma$ be a closed, embedded, locally convex surface with non-degenerate
shape operator in a hyperbolic 
end $E$. In a first-order deformation of $E$ and $\Sigma$, we have on $\Sigma$
$$ H' + \frac 12\langle I', \II\rangle_I da_I = \frac 12\langle \III',
\II - H^*\III\rangle_{\III}da_{\III}~, $$
where $\III$ is the third fundamental form of $\Sigma$ and $H^*=H/(K+1)$
is the curvature of the dual surface.
\end{lemma}

\begin{proof}
By definition, we have $\III=I(B\cdot, B\cdot)$, where $B$ is the shape operator
of $\Sigma$. Let $B^*=B^{-1}$ and let $Id$ denote the identity, then
$$ \II-H^*\III = \III((B^*-\tr(B^*)Id)\cdot, \cdot) = \III\left(\frac B{\det B}\cdot,
\cdot\right)~. $$
Let $A:T\Sigma\to T\Sigma$ be the self-adjoint (for $I$) bundle morphism such that
$I'=I(A\cdot, \cdot)$. Then a simple computation shows that 
$$ \III' = \III((B^{-1}AB + B^{-1}B' + (B^{-1}B')^*)\cdot, \cdot)~, $$
where the $*$ is the adjoint with respect to $\III$. Therefore
$$ \langle \III', \II - H^*\III\rangle_{\III} = 
\frac {\tr((B^{-1}AB + B^{-1}B' + (B^{-1}B')^*)B)}{\det B}~. $$
Since $B^*=B$, it follows that 
$$ \langle \III', \II - H^*\III\rangle_{\III} = 
\frac {\tr(AB + 2B')}{\det B}~. $$
But $da_{\III} = \det(B)da_I$, so it follows that
$$ \langle \III', \II - H^*\III\rangle_{\III} da_{\III} = 
\tr(AB + 2B') da_I = (2H'+\langle I', \II\rangle_I) da_I~, $$
as needed.
\end{proof}

\begin{proof}[Proof of Proposition \ref{pr:duality}]
Let $E\in \cHE$ be a hyperbolic end, and let $M\in \cGH_1$ be the 
dual GHM de Sitter manifold. Thanks to Proposition \ref{pr:same-hyp}
and Proposition \ref{pr:same-ds}, we only need to prove the statement
for any arbitrary value of $H$ and $H_*$, so we suppose (without loss
of generality) that $\Sigma^*_{H_*}$ is on the positive side of $\Sigma_H$.

We denote by $\Omega$ the domain of $M$ bounded by $\Sigma_H$ and
$\Sigma^*_{H_*}$. We then define
$$ W = V(\Omega) + \frac 12\int_{\Sigma_H} H da_I = V(\Omega)- \frac 12 HA(\Sigma_H)~. $$
It then follows from Lemma \ref{lm:schlafli} and from Corollary \ref{cr:schlafli}
that, in a first-order deformation of $M$,
$$ 2W' = \int_{\Sigma^*_{H_*}} H'_* + \frac 12 \langle I',\II\rangle da_I 
- \int_{\Sigma_H} \frac 12\langle I',\II-HI\rangle da_I~. $$
(The sign differs from that of Corollary \ref{cr:schlafli} because of the orientation
on $\Sigma_H$.)

Using Lemma \ref{lm:duality}, we can reformulate this equation as
$$ 2W' = \int_{\Sigma^*_{H_*}}\frac 12 \langle \III',\II-H_*\III\rangle da_{\III} 
- \int_{\Sigma_H} \frac 12\langle I',\II-HI\rangle da_I~. $$
Now the duality between $\bbH^3$ and $dS^3$ exchanges the induced metric and the
third fundamental forms of surfaces, and the equation becomes
\begin{eqnarray*}
2W' & = & \int_{\Sigma_{H_*}}\frac 12 \langle I',\II-H_*I\rangle da_I 
- \int_{\Sigma_H} \frac 12\langle I',\II-HI\rangle da_I \\
& = & \int_{\Sigma_{H_*}}\frac 12 \langle I',\II_0\rangle da_I - \frac {H_*}2 A(\Sigma_{H_*})'  
- \int_{\Sigma_H} \frac 12\langle I',\II_0\rangle da_I + \frac H2 A(\Sigma_H)'~.
\end{eqnarray*}
This means that
$$ d(2W + \frac {H_*}2 A(\Sigma_{H_*}) - \frac H2 A(\Sigma_H)) =
(\cmc^{Hyp}_{H})^*\theta - (\cmc^{dS}_{H_*})^*\theta~, $$
where $\theta$ denotes again the Liouville form of $T^*\cT$. The
result follows by taking the exterior differential of this last equation.
\end{proof}

We can now prove Theorem \ref{tm:wick-ds}. 

\begin{proof}[Proof of Theorem \ref{tm:wick-ds}]
Let $H\in (-2,2)$ and $H_*\in (-\infty, -2)$,
then it follows from the definition of $W^{dS}_{H,H_*}$ that 
$$ W^{dS}_{H,H_*} = (\cmc^{dS}_{H_*})^{-1}\circ \cmc^{Hyp}_H~. $$
The statement therefore follows directly from Proposition \ref{pr:duality},
along with Theorem \ref{tm:loustau-min}.
\end{proof}

\subsection{Minkowski Wick rotations and Wick rotations between 
moduli spaces of Lorentzian space-times}
We do not elaborate here on the symplectic properties of Wick rotations between 
quasifuchsian manifolds and GHM Minkowski manifolds. Note that there are at least
two natural Wick rotations one can consider:
\begin{itemize}
\item The map $W^{mink}_{H,H'}:\cAF'\to \cGH_0$, depending on  the choice of 
$H\in (-2,2)$ and of $H'\in (-\infty,0)$ sending an almost-Fuchsian manifold
$M\in \cAF'$ containing a CMC-$H$ surface $\Sigma_H$ to the unique GHM Minkowski
containing a CMC-$H'$ surface with the same data $([I],\II_0)$ as $\Sigma$.
(This map is well-defined by \cite[Lemma 6.1]{minsurf}.)
\item The map sending a hyperbolic end $E$ with boundary data $(m,l)\in \cT\times \cML$
on its pleated surface to the GHM Minkowski manifold for which $(m,l)$ describes the
initial singularity (see \cite{mess}). 
\end{itemize}
It would be interesting to know whether those maps have interesting properties
related to the natural symplectic structures on $\cAF'$ (resp. $\cHE$) and on 
$\cGH_0$.

As a final note, we have considered here only Wick rotations between hyperbolic manifolds and constant
curvature Lorentzian space-times --- either AdS, de Sitter or Minkowski. However a number
of statements on ``Wick rotations'' between constant curvature Lorentzian space-times
of different types (AdS to Minkowski, etc) clearly follow by composing different maps.
We leave the details to the interested reader.

%%% Local Variables: 
%%% mode: latex
%%% TeX-master: "doublemaps"
%%% End: 